\documentclass[12pt]{amsart}

\usepackage{fullpage}
\usepackage{mathtools}
\usepackage{amssymb,amsmath,amsthm,amscd,mathrsfs,graphicx}
\usepackage[dvipsnames]{xcolor}
\usepackage{bm}
\usepackage{dsfont}
\usepackage{enumerate}
\usepackage{epsfig}
\usepackage{float}
\usepackage{latexsym, amsxtra}
\usepackage{pdfpages}
\usepackage{multicol}
\usepackage[normalem]{ulem}
\usepackage{psfrag}
\usepackage{stmaryrd}
\usepackage{tikz}
\usepackage[T1]{fontenc}
\usepackage{url}
\usepackage{verbatim}
\usepackage{xy}
\usepackage{indentfirst}
\usepackage{tikz-cd}
\tikzset{%
    symbol/.style={%
        ,draw=none
        ,every to/.append style={%
            edge node={node [sloped, allow upside down, auto=false]{$#1$}}}
    }
}
\usepackage{longtable}

\usepackage[
  colorlinks=true,
  linkcolor=red,
  urlcolor=blue,
  citecolor=blue
]{hyperref}

\usepackage{multirow}
\usepackage{algorithm}
\usepackage{algpseudocode}

\makeatletter
\def\thm@space@setup{%
  \thm@preskip=2ex \thm@postskip=2ex
}
\makeatother

\makeatletter

\newcommand{\Rmnum}[1]{\expandafter\@slowromancap\romannumeral #1@}
\makeatother

\numberwithin{equation}{section}
\theoremstyle{plain}

\newtheorem{thm}{Theorem~}[section]
\newtheorem{lem}[thm]{Lemma~}

\newtheorem{prop}[thm]{Proposition~}

\newtheorem{cor}[thm]{Corollary~}

\newtheorem{prop-def}[thm]{Proposition-Definition~}

\theoremstyle{remark}
\newtheorem{rmk}[thm]{Remark~}

\theoremstyle{definition}
\newtheorem{defn}[thm]{Definition~}

\newcommand{\calM}{\mathcal{M}}

\newcommand{\calH}{\mathcal{H}}
\newcommand{\calP}{\mathcal{P}}
\newcommand{\calC}{\mathcal{C}}
\newcommand{\calF}{\mathcal{F}}
\newcommand{\calL}{\mathcal{L}}

\newcommand{\CC}{\mathbb{C}}
\newcommand{\ZZ}{\mathbb{Z}}
\newcommand{\RR}{\mathbb{R}}

\newcommand{\PP}{\mathbb{P}}
\newcommand{\FF}{\mathbb{F}}
\newcommand{\QQ}{\mathbb{Q}}

\newcommand{\DD}{\mathbb{D}}

\newcommand\PGL{\mathrm{PGL}}

\newcommand\PSL{\mathrm{PSL}}

\newcommand\id{\mathrm{id}}

\newcommand\disc{\mathrm{disc}}
\newcommand\rank{\mathrm{rank}}

\newcommand\SL{\mathrm{SL}}

\newcommand\GL{\mathrm{GL}}

\newcommand\ord{\mathrm{ord}}

\newcommand{\Aut}{\mathrm{Aut}}

\newcommand{\bs}{\backslash}
\newcommand{\dbs}{\bs\hspace{-0.5mm}\bs}

\title{Automorphism Groups of Smooth Cubic Fourfolds through Lattice Theory}

 \author[J. Fu]{Jie Fu}
\address{Tsinghua University, China}
\email{fu-j21@mails.tsinghua.edu.cn}

 \author[Z. Zheng]{Zhiwei Zheng}
\address{Tsinghua University, China}
\email{zhengzhiwei@mail.tsinghua.edu.cn}

\date{}

\begin{document}
\bibliographystyle{amsalpha}

\begin{abstract}
Laza and the second author have classified all possible symplectic automorphism groups of smooth cubic fourfolds. There are $34$ such groups, and by Koike, there are $48$ families of cubic fourfolds with speficied symplectic automorphism group. For $42$ among those families, we complete the classification of all possible automorphism groups. The approach of this paper is mainly lattice theoretic, starting from known results about the period map of cubic fourfolds by Voisin, Hassett, Looijenga and Laza. We use a lattice-enumeration algorithm implemented in OSCAR.
\end{abstract}

\maketitle

 \setcounter{tocdepth}{1}
	\tableofcontents

\emph{Notation}:
\begin{enumerate}
\item For groups $G_1$ and $G_2$, we write $G_1<G_2$ if $G_1$ is a subgroup of $G_2$.
\item If $G_{1}<G_{2}$, we denote the centralizer and normalizer of $G_{1}$ in $G_{2}$ by $C_{G_{2}}(G_{1})$ and $N_{G_{2}}(G_{1})$, respectively.
\item We write $G_1:G_2$ for a semidirect product with normal subgroup $G_1$ and quotient $G_2$. We write $G_{1}.G_{2}$ for an extension that fits into an exact sequence
\begin{align*}
    1\longrightarrow G_{1}\longrightarrow G_{1}.G_{2}\longrightarrow G_{2}\longrightarrow 1.
\end{align*}
\item We use $L_2(q)$ to denote $\PSL_2(\FF_q)$.
\item For a positive integer $n$, we denote by $\mu_{n}$ the group of $n$'th roots of unity in $\CC^{\times}$.
\item We use $3^{1+4}$ to denote $3^{1+4}_+$, the extraspecial $3$-group of order $243$ and exponent $3$.
\item For $g\in \GL(n,\CC)$, we write $\overline{g}$ for its image under the natural projection $\GL(n,\CC)\rightarrow \PGL(n,\CC)$.
\item For a lattice $L$ and a field $\FF\in\{\QQ,\RR,\CC\}$, we write $L_\FF=L\otimes_\ZZ\FF$.
\item We use $A_n, D_n,E_n$ to denote negative root lattice of corresponding Dynkin diagram.
\end{enumerate}

\section{Introduction}

For a smooth cubic fourfold $X\subset \PP^5=\PP^5_\CC$, a basic problem is to determine its automorphism group $\Aut(X)$. Its symplectic subgroup $\Aut^s(X)$ consists of the automorphisms acting trivially on $H^{3,1}(X)$. Laza and Zheng classified all possible symplectic automorphism groups using the global Torelli theorem, the image of period map, and lattice theory \cite{laza2022automorphisms}. Koike later described the corresponding projective representations and invariant cubic forms \cite{KOIKE202512,KOIKE2026}. Yang--Yu--Zhu classified the finite groups that admit faithful actions on smooth cubic fourfolds \cite{yang2024automorphism}.

This paper continues \cite{FWZ26}. The previous paper gives general restrictions on non-symplectic automorphisms, determines all rank-$19$ cases, and treats some additional families. Here we give a uniform lattice enumeration for the connected symplectic families with $15\leq\rank(S)\leq19$. The rank-$20$ families are zero-dimensional and were already determined in \cite[Theorem~1.8]{laza2022automorphisms}; see also \cite{KOIKE202512,KOIKE2026}. We do not obtain new results about the classification problem for $\rank(S)<15$, or equivalently $\Aut^s(X)\cong C_1, C_2, C_3, C_2^2, C_4$, or $S_3$.

For a smooth cubic fourfold $X$, its \emph{non-symplectic index} is
\[
    m(X)=[\Aut(X):\Aut^s(X)].
\]
The action on the one-dimensional space $H^{3,1}(X)$ gives an exact sequence
\[
1\longrightarrow\Aut^s(X)\longrightarrow\Aut(X)
\longrightarrow\mu_{m(X)}\longrightarrow1.
\]
Thus the index determines the cyclic quotient, but it does not by itself determine the abstract group $\Aut(X)$. 

Let $\Lambda_0$ be the primitive middle integral cohomology of $X$. Let $G=\Aut^s(X)$. Write
\[
    S=S_G(X),\qquad T=\Lambda_0^G
\]
for the coinvariant and invariant lattices in $\Lambda_0$. Based on \cite{hohn2016290} \cite{laza2022automorphisms} \cite{FWZ26}, one can enumerate all possible pairs $(S,T)$ arising from smooth cubic fourfolds. Note that for a given abstract pair $(S,T)$, different embeddings of $S\oplus T$ into $\Lambda_0$ may correspond to different families of cubic fourfolds.

The main result of this paper is the following:

\begin{thm}\label{thm:main}
Let $X$ be a smooth cubic fourfold, let $G=\Aut^s(X)$, and let $S=S_G(X)$ be its coinvariant lattice. If $\rank(S)\geq15$, then all possible groups $\Aut(X)$ and indices $m(X)$ are classified in Table~\ref{table: main}. 
\end{thm}

This theorem is proved through a lattice-theoretic argument and the main calculations are carried out in OSCAR \cite{OSCAR}, an algebra-computation software. We are also inspired by the work \cite{Brandhorst_Hofmann_2023}, which gives the classification of finite groups acting on complex K3 surfaces. For each lattice pair $(S,T)$, we enumerate finite-order isometries $f_T\in O(T)$ and their equivariant extensions to $\Lambda_0$. We then apply the signature, root, and discriminant form tests. If $f_T$ survives after those tests, then it arises from geometry of smooth cubic fourfolds. If $f_{\Lambda_0}\in\widetilde O(\Lambda_0)$ is the resulting isometry and $f_S=f_{\Lambda_0}|_S$, the algorithm computes
\[
\langle\widetilde O(S),f_S\rangle<O(S),
\]
which turns out to be an automorphism group of a smooth cubic fourfold.
%By \cite[Lemma~6.4 and Proposition~6.5]{laza2022automorphisms}, the action of $\Aut(X)$ on $S$ is faithful when $\rank(S)\ge 13$. Thus there is a family of cubic fourfolds with generic automorphism group $A_S(f)$. 

The paper is organized as follows. Section~2 reviews the Hodge-theoretic and lattice-theoretic background. Section~3 discusses moduli spaces with specified group action. Section~4 gives lattice criteria underlying the algorithms. Section~5 describes the enumeration algorithms. Section~6 presents the results and proves Theorem~\ref{thm:main}.

\textbf{Acknowledgements:}
We thank Xun Yu for his interests and helpful suggestions. We thank Shihao Wang for related discussions. We thank Stevell Muller for help with issues about OSCAR system.
The mathematical contents and algorithms are carried out by human. We used AI to assist with coding and manuscript polishing. We have reviewed all AI-assisted codes and take full responsibility for the content of the paper and the accompanying computational files.

\section{Preliminaries}

\subsection{Hodge Structures and the Global Torelli Theorem}

Let $X\subset\PP^5$ be a smooth cubic fourfold. We denote its automorphism group by $\Aut(X)$. The subgroup of automorphisms acting trivially on $H^{3,1}(X,\CC)$ is denoted by $\Aut^s(X)$. By Matsumura--Monsky \cite{Matsumura1963OnTA}, every automorphism of a smooth cubic fourfold is linear, and $\Aut(X)$ is finite.

The lattice $\Lambda:=H^4(X,\ZZ)\cong I_{21,2}$ is odd and unimodular. Let $h\in H^2(X,\ZZ)$ be the hyperplane class, and let $\Lambda_0$ be the orthogonal complement of $h^2$ in $\Lambda$. Then
\[
\langle h^2\rangle\cong\langle3\rangle,
\qquad
\Lambda_0\cong E_8^{\oplus2}\oplus U^{\oplus2}\oplus A_2.
\]
The Hodge numbers are $h^{4,0}=h^{0,4}=0$, $h^{3,1}=h^{1,3}=1$, and $h^{2,2}=21$.

The discriminant group is $A_{\Lambda_0}\cong \ZZ/3\ZZ$. The type IV domain associated with $\Lambda_0$ is
\[
\DD=\PP\{x\in\Lambda_{0,\CC}\mid \langle x,x\rangle=0,\ \langle x,\overline{x}\rangle<0\}^\circ,
\]
which means one connected component of the space of isotropic lines satisfying the indicated positivity condition. The other component is $\overline{\DD}$. Complex conjugation exchanges the two components.

We use the following subgroups of $O(\Lambda_0)$:
\begin{itemize}
    \item $O(\Lambda_0)$ is the full isometry group of $\Lambda_0$;
    \item $\widetilde{O}(\Lambda_0)$ is the subgroup acting trivially on $A_{\Lambda_0}$;
    \item $O^+(\Lambda_0)$ is the subgroup preserving the component $\DD$;
    \item $\Gamma=O^*(\Lambda_0)=\widetilde{O}(\Lambda_0)\cap O^+(\Lambda_0)$.
\end{itemize}
We have inclusions
\[
O^*(\Lambda_0)<\widetilde{O}(\Lambda_0)<O(\Lambda_0),
\qquad
O^*(\Lambda_0)<O^+(\Lambda_0)<O(\Lambda_0),
\]
and each consecutive inclusion has index $2$.

A vector $r\in\Lambda_0$ is a \emph{short root} if $\langle r,r\rangle=2$. It is a \emph{long root} if $\langle r,r\rangle=6$ and $\langle r,\Lambda_0\rangle=3\ZZ$. These roots define the $\Gamma$-invariant hyperplane arrangements
\begin{align*}
\calH_6 &=\{[x]\in \DD\sqcup\overline{\DD}\mid x\perp r\text{ for some short root }r\},\\
\calH_2 &=\{[x]\in \DD\sqcup\overline{\DD}\mid x\perp r\text{ for some long root }r\}.
\end{align*}

Let $\calM$ be the moduli space of smooth cubic fourfolds. The period line $H^{3,1}(X)\subset\Lambda_{0,\CC}$ defines the period map $\calM\to\DD/\Gamma$. The global Torelli theorem was proved in \cite{voisin1986torelli,voisin2008erratum} and \cite{looi2009period}. We use the following form (see also \cite[Proposition~2.1]{zheng2021orbifold} and \cite[Proposition~2.4]{laza2022automorphisms} for related discussion):

\begin{thm}
\label{theorem: global torelli}
Let $X$ and $X^{\prime}$ be smooth cubic fourfolds. Any Hodge isometry $H^4(X^{\prime},\ZZ)\xrightarrow{\sim} H^4(X,\ZZ)$ preserving the square of the hyperplane class must be induced by a unique isomorphism $X\xrightarrow{\sim} X^{\prime}$.
\end{thm}

\begin{cor}
    For every smooth cubic fourfold $X$, there is a canonical isomorphism
    \[
    \Aut(X)\cong O_{HS}(\Lambda, h^2),
    \]
    where the group on the right consists of the automorphisms of $\Lambda$ that preserve the Hodge decomposition and $h^2$.
\end{cor}

The image of the period map is described by the following theorem.

\begin{thm}[{\cite{laza2010period,looi2009period}}]\label{thm:global_Torelli_image}
    The period map for cubic fourfolds gives an isomorphism of quasi-projective varieties
    \[
    \calP:\calM\to (\DD\backslash(\calH_2\cup\calH_6))/\Gamma.
    \]
\end{thm}

\begin{rmk}
    Since $[O^+(\Lambda_0):O^*(\Lambda_0)]=[O(\Lambda_0):\widetilde{O}(\Lambda_0)]=2$ and $-\id_{\Lambda_0}\in O^+(\Lambda_0)\setminus O^*(\Lambda_0)$, we have
    \[
    O^+(\Lambda_0)=\langle O^*(\Lambda_0),-\id_{\Lambda_0}\rangle,
    \qquad
    O(\Lambda_0)=\langle\widetilde{O}(\Lambda_0),-\id_{\Lambda_0}\rangle.
    \]
    Therefore
    \[
        (\DD\backslash(\calH_2\cup\calH_6))/\Gamma=(\DD\backslash(\calH_2\cup\calH_6))/O^+(\Lambda_0),
    \]
    which also equals
    \[
    ((\DD\sqcup \overline{\DD})\backslash(\calH_2\cup\calH_6))/\widetilde{O}(\Lambda_0)=((\DD\sqcup \overline{\DD})\backslash(\calH_2\cup\calH_6))/O(\Lambda_0).
    \]
\end{rmk}

\subsection{Lattice Theory}

By a lattice we mean a free $\ZZ$-module of finite rank equipped with an integer-valued nondegenerate symmetric bilinear form.

\subsubsection{The Discriminant Form}

For a lattice $L$, its \emph{discriminant} $\disc(L)$ is the determinant of a Gram matrix of $L$. The lattice is \emph{unimodular} if $\disc(L)=\pm1$. Its \emph{dual lattice} is
\[
L^\vee=\{x\in L_\QQ\mid \langle x,y\rangle\in\ZZ\text{ for every }y\in L\},
\]
and its \emph{discriminant group} is the finite abelian group $A_L:=L^\vee/L$. The \emph{length} $\ell(A_L)$ is the minimal number of generators of $A_L$. Every lattice has a discriminant bilinear form. If $L$ is even, it also has a discriminant quadratic form:
\begin{itemize}
    \item $b_L:A_L\times A_L\to\QQ/\ZZ$, defined by $(x_1+L,x_2+L)\mapsto \langle x_1,x_2\rangle+\ZZ$;
    \item $q_L:A_L\to \QQ/2\ZZ$, defined by $x+L\mapsto \langle x,x\rangle+2\ZZ$.
\end{itemize}
We use Conway--Sloane symbol \cite[Chapter 15]{conway1999sphere} \cite[Appendix A]{laza2022automorphisms} for discriminant quadratic forms.

\begin{rmk}
    In \cite{Brandhorst_Hofmann_2023} and related references, the term \emph{integral lattice} is used to distinguish these lattices from non-integral or hermitian lattices, and the notation $D_L$ is used in place of $A_L$.
\end{rmk}

For an even lattice $L$, let $O(q_L)$ be the group of automorphisms of $A_L$ preserving $q_L$. There is a natural group homomorphism $O(L)\to O(q_L)$. We denote by $\widetilde{O}(L)$ the kernel of this homomorphism.

If the bilinear form on $L_1$ is the restriction of the form on $L_2$, we call $L_1$ a \emph{sublattice} of $L_2$. We call $L_2$ an \emph{overlattice} of $L_1$ if $L_2/L_1$ is finite. The sublattice $L_1$ is \emph{primitive} in $L_2$ if $L_2/L_1$ is free.

A $G$-action on a lattice $L$ means a linear action preserving the bilinear form. It induces an action on $A_L$, and this action preserves $q_L$ when $L$ is even. We define the invariant and coinvariant lattices by
\[
L^G:=\{x\in L\mid gx=x\text{ for all }g\in G\},
\qquad
S_G(L):=(L^G)^\perp_L.
\]
Both are primitive sublattices of $L$. When $L=\Lambda_0$ is the primitive cohomology lattice of a cubic fourfold $X$, we also write $S_G(X)$ for $S_G(\Lambda_0)$.

\subsubsection{Relation between Extensions and Discriminant Forms}

Assume that $L_1$ and $L_2$ are even lattices and $L_2$ is an overlattice of $L_1$. From the chain of inclusions $L_1\subset L_2\subset L_2^\vee\subset L_1^\vee$, the subgroup $H_{L_2}=L_2/L_1\subset A_{L_1}$ is isotropic, and $L_2^\vee/L_1=(H_{L_2})^\perp\subset A_{L_1}$. The resulting correspondence is stated below.
\begin{thm}[{\cite[Proposition~1.4.1]{Nikulin_1980}}]
    Fix an even lattice $L_1$. Then $L_2\mapsto H_{L_2}$ gives a bijection between even overlattices of $L_1$ and isotropic subgroups of $A_{L_1}$. The discriminant form on $L_2$ is given by
    \[
    q_{L_2}=\left.\left(q_{L_1}|_{(H_{L_2})^\perp}\right)\right/H_{L_2}.
    \]
\end{thm}

\begin{rmk}
    In the theorem, $L_2$ is identified as a sublattice of $L_1^\vee$, while $H_{L_2}$ is identified as a subgroup of $A_{L_1}$.
\end{rmk}

Suppose $L_1=S_1\oplus S_2\subset L_2$ and both $S_1$ and $S_2$ are primitive in $L_2$. Then
\[
H_{L_2}\hookrightarrow A_{L_1}=A_{S_1}\oplus A_{S_2},
\]
and the two projections $p_{S_i}:H_{L_2}\to A_{S_i}$ are injective. Let their images be $H_{L_2,S_i}$. The map
\[
p_{S_1}\circ p_{S_2}^{-1}:H_{L_2,S_2}\longrightarrow H_{L_2,S_1}
\]
is an anti-isometry, that is, an isomorphism of finite abelian groups that changes the sign of the quadratic form. This gives the following correspondence.
\begin{thm}[{\cite[Proposition~1.5.1 and Corollary~1.5.2]{Nikulin_1980}}]\label{thm:double_equivariant_discriminant}
    \begin{enumerate}
        \item For even lattices $S_1$ and $S_2$, an even overlattice $L\supset S_1\oplus S_2$ such that $S_1$ and $S_2$ are primitive in $L$ is determined by subgroups $H_1\subset A_{S_1}$ and $H_2\subset A_{S_2}$, together with an anti-isometry $\gamma:H_1\xrightarrow{\sim} H_2$.
        \item Suppose the data $(H_1,H_2,\gamma)$ and $(H_1',H_2',\gamma')$ define primitive embeddings $S_1\hookrightarrow L$ and $S_1\hookrightarrow L'$. The embeddings are isomorphic by an isometry inducing the identity on $S_1$ if and only if $H_1=H_1'$ and $\gamma=\overline{\psi}\circ\gamma'$ for some $\psi\in O(S_2)$. There is an isometry $L\xrightarrow{\sim}L'$ preserving $S_1$ if and only if
        \[
        \gamma\circ\overline{\varphi}=\overline{\psi}\circ\gamma'
        \]
        for some $\varphi\in O(S_1)$ and $\psi\in O(S_2)$.
    \end{enumerate}
\end{thm}

\subsubsection{Enumeration of Equivariant Extensions}

For a lattice $L$ equipped with an isometry $f_L$, set
\[
    G(L,f_L):=\operatorname{im}\bigl(C_{O(L)}(f_L)\to O(q_L)\bigr)
    \subset C_{O(q_L)}(\overline{f_L}),
\]
where $\overline{f_L}$ is the induced action on $A_L$. When no confusion is likely, we write $G_L$ for $G(L,f_L)$.

Let $(M,f_M)$ and $(N,f_N)$ be even lattices equipped with isometries. Consider primitive extensions
\[
M\oplus N\subset L
\]
such that both $M$ and $N$ are primitive in $L$, and such that
$f_M\oplus f_N$ extends to an isometry of $L$. By the correspondence
between even overlattices and isotropic subgroups of the discriminant
form, such an extension is determined by an isotropic subgroup
\[
H\subset A_M\oplus A_N.
\]
The lattices $M$ and $N$ remain primitive in $L$ exactly when the two projections of $H$ to $A_M$ and $A_N$ are injective. In this case, $H$ is the graph of an anti-isometry $\gamma:H_M\xrightarrow{\sim}H_N$. The extension is equivariant with respect to $f_M\oplus f_N$ if and only if
\[
\gamma\circ \overline f_M|_{H_M}
=
\overline f_N|_{H_N}\circ \gamma .
\]

Let $\mathcal G(M,N)$ be the finite set of these equivariant anti-isometries. The group $G_N$ acts on the left, and $G_M$ acts on the right; their actions also transform the image and domain subgroups. The isomorphism classes of equivariant primitive extensions that preserve the two factors $(M,f_M)$ and $(N,f_N)$ are therefore represented by the finite double quotient
\[
G_N\backslash \mathcal G(M,N)/G_M .
\]

Thus the extension problem reduces to finite data on the discriminant groups. The remaining step is to compute the image $\operatorname{im}(C_{O(L)}(f_L)\to O(q_L))$.

\subsubsection{Computing $\operatorname{im}(C_{O(L)}(f_L)\to O(q_L))$}

We briefly recall from \cite{Brandhorst_Hofmann_2023} how to compute
\[
    G_L=\operatorname{im}\bigl(C_{O(L)}(f_L)\to O(q_L)\bigr).
\]
This algorithm uses the implementation in OSCAR \cite{OSCAR}.

If $L$ is definite, then $O(L)$ is finite and can be computed directly.
Suppose that $L$ is indefinite of rank at least $3$ and
$f_L=\pm\id$. Miranda--Morrison theory describes the cokernel of
\[
    O(L)\longrightarrow O(q_L)
\]
in terms of local determinant and spinor-norm data. Only the primes
dividing $\disc(L)$ contribute, and strong approximation underlies the
resulting local-to-global criterion; see
\cite{Kneser_1966_StrongApproximation,Brandhorst_Hofmann_2023}.
Every indefinite lattice occurring in this step has signature $(r,2)$
with $r\geq1$, so the rank hypothesis is satisfied. We use the resulting
finite image in $O(q_L)$ and do not need the explicit obstruction
sequence.

If $f_L\neq\pm\id$, its cyclotomic components carry natural Hermitian
lattice structures, and the centralizer of $f_L$ is described by the
corresponding unitary groups. For an irreducible cyclotomic component,
the image on the discriminant form is determined by local unitary-group
computations; see \cite{Brandhorst_Hofmann_2023}.

\subsection{Finite-Order Isometries}

Let $L$ be an even lattice, and let $f\in O(L)$ have order $m$. For each divisor $d$ of $m$, set $L_d=\ker(\Phi_d(f))$, where $\Phi_d$ is the $d$'th cyclotomic polynomial. Then there is an inclusion of lattices
\[
    \bigoplus_{d\mid m}L_d\subset L,
\]
whose two sides have the same rank.

On $\ker\Phi_m(f)$, one further has the real decomposition
\[
    \ker\Phi_m(f)\otimes\RR
    =
    \bigoplus_{k\in(\ZZ/m\ZZ)^\times/\{\pm1\}}
    \ker\bigl(f+f^{-1}-\zeta_m^k-\zeta_m^{-k}\bigr),
\]
where each summand carries the restriction of the bilinear form.

For any fixed finite order, there are only finitely many conjugacy classes of isometries of $L$, whether $L$ is definite or indefinite. See \cite[\S4.1]{Brandhorst_Hofmann_2023} and \cite{Fritz}.

\begin{thm}
    Given an even lattice $L$ and a positive integer $m$, the set of conjugacy classes of elements
    $f\in O(L)$ with $f^m=\id$ is finite.
\end{thm}

An algorithm for enumerating finite-order isometries was developed in \cite{Brandhorst_Hofmann_2023} and implemented in OSCAR \cite{OSCAR}. We briefly review it.

Each $L_d$ is naturally a hermitian lattice over $\ZZ[\zeta_d]$, where $\zeta_d$ acts as $f|_{L_d}$. Once $L$ is fixed, only finitely many genera of the $L_d$ can occur. These genera are determined by their ranks, signatures, determinants, and local invariants. Each genus contains only finitely many isometry classes, so representatives for all possible combinations can be enumerated. After passing to the underlying integral lattices, these representatives carry the required actions.

Finally, one uses the inclusion $\bigoplus_{d\mid m}L_d\subset L$ and the equivariant-extension algorithm above to enumerate the possible gluing data successively. The resulting actions on $L$ give the desired representatives.

\section{Moduli of Cubic Fourfolds with Group Actions}\label{sec:prescribed-automorphisms}

\subsection{GIT Construction}\label{subset:GIT}

Let $A$ be a finite group acting faithfully on $\PP^5$ through
\[
    \rho:A\hookrightarrow\PGL(6,\CC).
\]
Assume that this action preserves a smooth cubic fourfold. Let $\widehat{\rho(A)}$ be the inverse image of $\rho(A)$ in $\SL(6,\CC)$, and let $\xi$ be the character by which $\widehat{\rho(A)}$ acts on the defining cubic form. The induced action on $H^{3,1}$ gives a character
\[
    \chi:A\longrightarrow\CC^\times.
\]
Its kernel $A^s$ is the symplectic subgroup. If $A=\Aut(X)$, then
\[
    1\longrightarrow\Aut^s(X)\longrightarrow\Aut(X)
    \xrightarrow{\chi}\mu_{m(X)}\longrightarrow1.
\]
For such an action, $[A:A^s]$ is its non-symplectic index. 

Following \cite[\S2]{yu2020moduli}, let $\mathcal V_{A,\rho,\xi}$ be the subspace of cubic forms on which $\widehat{\rho(A)}$ acts through $\xi$, and set
\[
    \mathcal F^{\mathrm{GIT}}_{A,\rho,\xi}
    =N_\rho^\xi(A)\dbs\PP\mathcal V_{A,\rho,\xi}^{\mathrm{sm}},
\]
where
\[
N_\rho^\xi(A)=\{g\in\SL(6,\CC)\mid
 g\widehat{\rho(A)}g^{-1}=\widehat{\rho(A)},\
 \xi(ghg^{-1})=\xi(h)\text{ for all }h\in\widehat{\rho(A)}\}.
\]
The forgetful map to the moduli space $\calM$ of smooth cubic fourfolds is finite \cite[Proposition~2.7]{yu2020moduli}. If $A=\Aut(X)$ for a general member, this map is the normalization of its image. In this paper, a \emph{family} associated with $(A,\rho,\xi)$ is a connected component of the image of this map in $\calM$. We denote a chosen component by $\calF_{A,\rho,\xi}$. Each component is irreducible: the smooth locus in the projectivized semi-invariant space is open in an irreducible variety, and irreducibility is preserved under the quotient and forgetful maps.

The corresponding period domain is defined using the $\chi$-eigenspace
\[
    (\Lambda_{0,\CC})^\chi
    =\{x\in\Lambda_{0,\CC}\mid a(x)=\chi(a)x\text{ for all }a\in A\}.
\]
Here $A$ acts through its induced representation on $\Lambda_0$.
If $\chi=\overline\chi$, the period domain is of type IV and has two connected components. Complex conjugation exchanges these components. If $\chi\neq\overline\chi$, the period domain is a connected complex ball. In this case, complex conjugation sends the $\chi$-eigenspace to the $\overline\chi$-eigenspace. Yu--Zheng identify this moduli space with an arrangement complement in the corresponding arithmetic quotient \cite[Theorem~1.1]{yu2020moduli}. This gives the following dimension formula.

\begin{prop}[{\cite[Theorem~1.1(i)]{yu2020moduli}}]\label{prop:moduli_dim}
The dimension of a connected family $\calF_{A,\rho,\xi}$ is
\[
\begin{cases}
\dim(\Lambda_{0,\CC}^\chi)-2,&\chi=\overline\chi,\\
\dim(\Lambda_{0,\CC}^\chi)-1,&\chi\neq\overline\chi.
\end{cases}
\]
\end{prop}

\subsection{Saturation of Families}

We use the saturation statement for symmetry types proved by Yu--Zheng \cite{yu2020moduli}.

\begin{lem}\label{lem:general-saturation}
Let $(A,\rho,\xi)$ be as above, and assume that its smooth locus is nonempty. After replacing it by an equivalent action, there is an action $(A',\rho',\xi')$ such that
\[
\rho(A)\subseteq\rho'(A'),
\qquad
\xi'|_{\widehat{\rho(A)}}=\xi,
\]
and the following properties hold.
\begin{enumerate}
\item The full automorphism action of a general member is $(A',\rho',\xi')$.
\item There is a finite morphism
\[
\mathcal F^{\mathrm{GIT}}_{A,\rho,\xi}
\longrightarrow
\mathcal F^{\mathrm{GIT}}_{A',\rho',\xi'},
\]
and the two forgetful morphisms to $\calM$ have the same image.
\end{enumerate}
The action $(A',\rho',\xi')$ is unique up to equivalence.
\end{lem}

\begin{proof}
Apply \cite[Propositions~2.5 and 2.8]{yu2020moduli} to the linear symmetry type defined by $(\widehat{\rho(A)},\xi)$. Proposition~2.8 constructs a larger symmetry type satisfying Condition~2.3, gives the finite morphism, and shows that the two images in the moduli space coincide. Proposition~2.5 shows that the larger group is the full automorphism group of a general member. Passing to the projective quotient gives the asserted action. At a general point of the common image, the full projective automorphism action is intrinsic. This proves uniqueness up to equivalence.
\end{proof}

\begin{rmk}\label{rmk:GIT-saturation}
The cited results are formulated for linear symmetry types. The lemma records only the projective, componentwise consequence used below.
\end{rmk}

We apply the following notions componentwise to a fixed connected family.

\begin{defn}\label{def:sat_id}
\begin{enumerate}
\item A family associated with $(A,\rho,\xi)$ is \emph{saturated} if, for a general member $X$, the subgroup $\rho(A)$ equals $\Aut(X)$ after a linear change of coordinates compatible with $\xi$. It is \emph{symplectically saturated} if the same statement holds for $\rho(A^s)$ and $\Aut^s(X)$.
\item Two actions $(A_i,\rho_i,\xi_i)$ are \emph{equivalent} if there exist an isomorphism $\theta:A_1\xrightarrow{\sim}A_2$ and $h\in\GL(6,\CC)$ such that
\[
\rho_2(\theta(g))=\overline h\rho_1(g)\overline h^{-1}
\]
for every $g\in A_1$. Conjugation by $h$ identifies the inverse images of the two projective groups in $\SL(6,\CC)$. We require
\[
\xi_2(h\widehat g h^{-1})=\xi_1(\widehat g)
\]
for every $\widehat g\in\widehat{\rho_1(A_1)}$.
\item Let $\tau_i:A_i\hookrightarrow\widetilde O(\Lambda_0)$ be the induced lattice actions. The two actions are \emph{lattice-theoretically equivalent} if there exist $\theta:A_1\xrightarrow{\sim}A_2$ and $h\in O(\Lambda_0)$ such that $\chi_1=\chi_2\circ\theta$ and
\[
\tau_2(\theta(g))=h\tau_1(g)h^{-1}
\]
for all $g\in A_1$.
\end{enumerate}
\end{defn}

Equivalent actions are lattice-theoretically equivalent. The converse need not hold. Indeed, a lattice-theoretic equivalence is not required to come from a linear change of coordinates. 

Complex conjugation sends an action and its defining cubics to their entrywise conjugates. It preserves the dimension, the non-symplectic index, and both saturation properties. In the real-character case, it exchanges the families arising from the two type IV components. The two families may still be equivalent in the above sense.

Saturatedness implies symplectic saturatedness. The next statement applies Lemma~\ref{lem:general-saturation} to a connected family arising from a group action. It is the componentwise saturation statement used in \cite[\S2]{yu2020moduli}.

\begin{prop-def}\label{prop-def:saturation}
Every connected family as above has a unique saturated action, up to equivalence, with the same image in $\calM$. It is called the \emph{saturation} of the family.
\end{prop-def}

\begin{proof}
Apply Lemma~\ref{lem:general-saturation} to the corresponding symmetry type and restrict the common image to the chosen connected family. The resulting action is the full automorphism action of a general member. It is unique up to equivalence because that full action is intrinsic. The corresponding GIT moduli space is the normalization of its image by \cite[Proposition~2.7]{yu2020moduli}.
\end{proof}

\begin{prop}\label{prop:lattice-equivalence}
Lattice-theoretically equivalent connected families have the same dimension and non-symplectic index. Their generic automorphism groups are isomorphic, and they have the same saturation properties. If the index is at least $3$, each lattice-theoretic equivalence class contains one equivalence class of connected families. If the index is $1$ or $2$, it contains at most two. In the latter case, complex conjugation exchanges the two possible classes. They coincide when either family is equivalent to its complex conjugate.
\end{prop}

\begin{proof}
Let $(\theta,h)$ define the lattice-theoretic equivalence. Since $O(\Lambda_0)=\langle\widetilde O(\Lambda_0),-\id_{\Lambda_0}\rangle$, after replacing $h$ by $-h$ if necessary, we may assume that $h\in\widetilde O(\Lambda_0)$. This replacement changes neither its conjugation action nor its action on the projectivized period domain. The isometry $h$ then extends to an isometry of $\Lambda$ fixing the square of the hyperplane class. It identifies the relevant character eigenspaces and their root hyperplane arrangements. Hence it identifies the corresponding arrangement complements in $\DD\sqcup\overline{\DD}$, possibly exchanging the two type IV components. By Theorem~\ref{thm:global_Torelli_image} and \cite[Theorem~1.1]{yu2020moduli}, these arrangement complements describe the corresponding families in $\calM$.

At a general period, conjugation by $h$ identifies the integral Hodge stabilizers. Theorem~\ref{theorem: global torelli} identifies these stabilizers with the projective automorphism groups of the cubic fourfolds. It follows that the generic automorphism groups are isomorphic and that the saturation properties agree.

If the index is at least $3$, the character is non-real and the period domain is a connected complex ball. Its arrangement complement therefore gives one equivalence class of connected families. If the index is $1$ or $2$, the character is real and the type IV domain has two components. Thus there are at most two classes, and complex conjugation exchanges them. If either class is equivalent to its conjugate, the two classes coincide.
\end{proof}

Thus the enumeration reduces the determination of the possible groups and indices to lattice-theoretic equivalence. For indices $1$ and $2$, deciding whether the two conjugate type IV components define equivalent families is a separate question. It is not needed for the result of this paper and will not be addressed below.

\section{Lattice Arguments}\label{sec:lattices}

\subsection{Symplectic Families}

Let $X$ be a smooth cubic fourfold and let $G<\Aut^s(X)$. Write
\[
    T=\Lambda_0^G,\qquad S=S_G(X)=(\Lambda_0^G)^\perp_{\Lambda_0}.
\]
When $G=\Aut^s(X)$, we call the connected family containing $X$ a \emph{symplectic family}, and we call $(S,T)$ its lattice pair. Its dimension is $20-\rank(S)$. The Laza--Zheng classification gives the possible lattices $S$ and identifies $G$ with
\[
    G\cong\widetilde O(S).
\]
The action of this group extends to $\Lambda_0$ by acting trivially on $T$ \cite{laza2022automorphisms}.

An \emph{embedded lattice datum} additionally records the primitive embedding $\iota:S\oplus T\hookrightarrow\Lambda_0$. We use the following consequence of the proof of the Laza--Zheng criterion.

\begin{lem}\label{lem:LZ-root-free}
For an embedded lattice datum occurring in the Laza--Zheng classification, the embedded coinvariant lattice $S\subset\Lambda_0$ contains neither short nor long roots. The absence of short roots follows from the Leech-pair condition; the absence of long roots follows from the argument in the proof of \cite[Theorem~4.5]{laza2022automorphisms}.
\end{lem}

We recall the following input from the Laza--Zheng and Koike classifications. There are $46$ abstract lattice pairs $(S,T)$ and $48$ connected symplectic families. Two of the lattice pairs give two connected families; every other pair gives one.

\begin{prop}\label{prop:ST-families}
For each of the $46$ lattice pairs $(S,T)$, the number of equivalence classes of connected symplectic families equals the number of lattice-theoretic equivalence classes. This number is one, except for $G\cong M_{10}$ and $G\cong S_{3,3}$, where it is two. Consequently, the Laza--Zheng data give $48$ connected symplectic families in total.
\end{prop}

\begin{proof}
Koike's classification and its corrigendum list the projective representations and invariant cubic spaces up to projective equivalence \cite{KOIKE202512,KOIKE2026}. Their comparison with the Laza--Zheng data is given in \cite[Corollary~4.9]{FWZ26}. The two representations for each of $M_{10}$ and $S_{3,3}$ give inequivalent embedded lattice data, whereas every other lattice pair occurs once. Moreover, \cite[Theorem~4.5 and its proof]{laza2022automorphisms} shows that every such embedded datum is root-free and defines a nonempty family. Hence the geometric and lattice-theoretic counts agree.
\end{proof}

\subsection{Non-symplectic Actions}

We first explain why it is enough to compute the action on the coinvariant lattice.

\begin{lem}\label{lem:faithful-S-action}
Let $X$ be a smooth cubic fourfold, and let
\[
S=S_{\Aut^s(X)}(X).
\]
If $\rank(S)\geq13$, then $\Aut(X)$ preserves $S$, and the restriction homomorphism
\[
r_S:\Aut(X)\longrightarrow O(S)
\]
is injective.
\end{lem}

\begin{proof}
By \cite[Lemma~6.4]{laza2022automorphisms}, $\Aut(X)$ preserves $S$. When $\rank(S)\geq13$, \cite[Proposition~6.5]{laza2022automorphisms} shows that the homomorphism
\[
\Aut(X)/\Aut^s(X)\longrightarrow O(q_S)
\]
induced by the action on $S$ is injective. Therefore an automorphism acting trivially on $S$ belongs to $\Aut^s(X)$. Since $\Aut^s(X)\cong\widetilde O(S)$ acts faithfully on $S$, this automorphism is the identity.
\end{proof}

Fix a smooth cubic fourfold $X$ in a symplectic family with lattice pair $(S,T)$, and set $m=m(X)$. Choose $f\in\Aut(X)$ whose image generates $\Aut(X)/\Aut^s(X)\cong\mu_m$, and write $f_T=f|_T$. After replacing $f$ by a power relatively prime to $m$, we may assume that $f$ acts on $H^{3,1}(X)$ by $\zeta_m$.

Set
\[
P=\ker\Phi_m(f_T),\qquad K=P^\perp_{\Lambda_0},\qquad f_P=f|_P.
\]
Then $P_\RR$ has signature $(*,2)$. Moreover, $K\subset H^{3,1}(X)^\perp$ contains neither short nor long roots.

Conversely, the image theorem for the cubic period map and strong global Torelli give the following construction.

\begin{prop}\label{prop:construct-family}
Suppose $f\in\widetilde O(\Lambda_0)$ preserves $S$ and $T$, the restriction $f_T$ has order $m$, and
\begin{enumerate}
\item $\ker(f_T+f_T^{-1}-(\zeta_m+\zeta_m^{-1})\id)\subset T_\RR$ has signature $(*,2)$;
\item $K=(\ker\Phi_m(f_T))^\perp_{\Lambda_0}$ contains neither short nor long roots.
\end{enumerate}
Then the $\zeta_m$-eigenspace contains periods of smooth cubic fourfolds. Each connected component of the resulting arrangement complement determines a family carrying the action of
\[
A=\langle f,\widetilde O(S)\rangle.
\]
Its dimension is
\[
\begin{cases}
\dim\ker(f_T-\zeta_m\id)-2,&m=1,2,\\
\dim\ker(f_T-\zeta_m\id)-1,&m\geq3.
\end{cases}
\]
\end{prop}

\begin{proof}
The signature condition defines either a type IV domain or a complex ball in the $\zeta_m$-eigenspace. Suppose that this domain is contained in a hyperplane orthogonal to a rational root. This root is orthogonal to the $\zeta_m$-eigenspace and, by Galois conjugation, to every primitive $m$-eigenspace. It therefore lies in $K$, contrary to condition (2). A general period avoids both forbidden arrangements. By Theorem~\ref{thm:global_Torelli_image}, it is the period of a smooth cubic fourfold.

Since $f\in\widetilde O(\Lambda_0)$, it extends to the cubic lattice and fixes $h^2$. The same is true for the extensions of $\widetilde O(S)$. All these isometries preserve the Hodge structure, so strong global Torelli realizes them as unique automorphisms. The dimension formula follows from Proposition~\ref{prop:moduli_dim}.
\end{proof}

Given an embedded pair $S\oplus T\hookrightarrow L$, a fitting extension of $f_T$ is an isometry $f_L\in O(L)$ that preserves $S$ and $T$ and restricts to $f_T$ on $T$. We call it stable if $f_L\in\widetilde O(L)$. For a stable fitting extension $f\in\widetilde O(\Lambda_0)$, write $f_S=f|_S$ and set
\[
A_S(f)=\langle\widetilde O(S),f_S\rangle<O(S).
\]

\begin{lem}\label{lem:group-from-output}
Assume that $\rank(S)\geq13$. Consider a symplectically saturated family of non-symplectic index $m$ constructed from a stable fitting extension $f\in\widetilde O(\Lambda_0)$. Restriction to $S$ identifies its acting group with
\[
A_S(f)=\langle\widetilde O(S),f_S\rangle.
\]
Moreover,
\[
\widetilde O(S)\triangleleft A_S(f),\qquad
A_S(f)/\widetilde O(S)\cong\mu_m,
\qquad
|A_S(f)|=m|\widetilde O(S)|.
\]
Lattice-theoretically equivalent outputs give isomorphic groups. If the family is saturated, then $A_S(f)$ is isomorphic to the automorphism group of a general member.
\end{lem}

\begin{proof}
Proposition~\ref{prop:construct-family} and strong global Torelli realize the lattice action by automorphisms. Lemma~\ref{lem:faithful-S-action} shows that restriction to $S$ is injective. Its image is generated by $\widetilde O(S)$ and $f_S$, and is therefore $A_S(f)$. The quotient acts faithfully on $H^{3,1}$, so it has order $m$.

A lattice equivalence conjugates the actions on $S$. The last assertion follows from the definition of saturation.
\end{proof}

\begin{rmk}\label{rmk:group-order-on-S}
The implementation computes $A_S(f)$ directly from the refined action on $S$. It also checks
\[
|A_S(f)|=m|\widetilde O(S)|.
\]
Thus the computed group is also verified by its order.
\end{rmk}

A fixed lattice action gives one equivalence class of connected families in the non-real-character case. In the real-character case, it gives at most two, one from each type IV component. Complex conjugation exchanges the two possible classes. They are lattice-theoretically equivalent and have the same numerical invariants. We denote representatives by $\calF_{(S,T,f)}^j$ and omit $j$ when there is only one class. The computation records their common lattice-theoretic class and does not decide whether the two geometric classes coincide.

\begin{prop}\label{prop:lattice_family_criteria}
\begin{enumerate}
\item The family $\calF_{(S,T,f)}^j$ is symplectically saturated if and only if
\[
|\widetilde O(K)|=|\widetilde O(S)|.
\]
\item Consider two symplectically saturated families
$\calF_{(S_1,T_1,f_1)}^{j_1}$ and
$\calF_{(S_2,T_2,f_2)}^{j_2}$.
They are lattice-theoretically equivalent if and only if there exists
$h\in O(\Lambda_0)$ such that
\[
h(P_1)=P_2,\qquad h f_{1,P}h^{-1}=f_{2,P}.
\]
\end{enumerate}
\end{prop}

\begin{proof}
\begin{enumerate}
\item Since $S\subset K$, every element of $\widetilde O(S)$ extends by the identity on $T$ and defines an element of $\widetilde O(K)$. Conversely, take $g_K\in\widetilde O(K)$ and let it act as the identity on $P$. Since $g_K$ acts trivially on $A_K$, the resulting action preserves the gluing subgroup of $K\oplus P\subset\Lambda_0$ and extends to an element of $\widetilde O(\Lambda_0)$. This extension fixes the period line and is therefore symplectic.

For each proper rational subspace $W\subsetneq P_\QQ$, the periods contained in $W_\CC$ form a proper analytic subset of the period domain. Indeed, if $W_\CC$ contained the full $\zeta_m$-eigenspace, then its Galois conjugates would contain every primitive $m$-eigenspace, and hence $W_\CC=P_\CC$. There are only countably many rational subspaces. By the Baire category theorem, their union does not contain the period domain. Thus, for a very general period, the smallest rational subspace containing the period line is $P_\QQ$.

An integral symplectic Hodge isometry fixes the period line and all its Galois conjugates, and therefore fixes $P$ pointwise. Its restriction to $K$ lies in $\widetilde O(K)$. Thus the symplectic automorphism group at a very general point, and hence the generic symplectic automorphism group, is $\widetilde O(K)$.

\item The forward implication follows by applying the lattice equivalence to the primitive eigenspace and then to all Galois conjugates. Conversely, suppose that $h(P_1)=P_2$ and $h f_{1,P}h^{-1}=f_{2,P}$. Then $h(K_1)=K_2$. By (1), the generic symplectic group of the $i$-th family is $\widetilde O(K_i)$, and its coinvariant lattice in $K_i$ is
\[
S_i=(K_i^{\widetilde O(K_i)})^\perp_{K_i}.
\]
Thus $h(S_1)=S_2$ and $h(T_1)=T_2$. The ambient isometry
\[
u=h^{-1}f_2^{-1}hf_1
\]
fixes $P_1$ pointwise and belongs to $\widetilde O(\Lambda_0)$; hence $u|_{K_1}\in\widetilde O(K_1)$. It follows that conjugation by $h$ identifies $A_{S_1}(f_1)$ with $A_{S_2}(f_2)$ and respects the character.

\end{enumerate}
\end{proof}

\section{The Enumeration Algorithm}\label{subsec:enumeration}

\subsection{Equivariant Extensions}

Let $L_1$ be an even positive-definite lattice, let $L_2$ be an even lattice, and let $f_{L_2}\in O(L_2)$ have order $m$. For an embedding
\[
\iota:L_1\oplus L_2\hookrightarrow\Lambda_0
\]
with both factors primitive, put
\[
P=\ker\Phi_m(f_{L_2}),\qquad L_3=\iota(P)^\perp_{\Lambda_0}.
\]
We say that $(L_1,L_2,f_{L_2},\iota)$ satisfies property $\clubsuit$ if the following conditions hold:
\begin{enumerate}
\item $f_{L_2}$ extends through $\iota$ to an element of $\widetilde O(\Lambda_0)$;
\item $\ker(f_{L_2}+f_{L_2}^{-1}-(\zeta_m+\zeta_m^{-1})\id)\subset (L_2)_\RR$ has signature $(*,2)$;
\item $L_3$ contains neither short nor long roots;
\item $|\widetilde O(L_3)|=|\widetilde O(L_1)|$.
\end{enumerate}
Its dimension is
\[
d(L_2,f_{L_2})=
\begin{cases}
\dim\ker(f_{L_2}-\zeta_m\id)-2,&m=1,2,\\
\dim\ker(f_{L_2}-\zeta_m\id)-1,&m\geq3.
\end{cases}
\]

Let $\calL(L_1,L_2,f_{L_2})$ be the set of embeddings satisfying property $\clubsuit$, modulo the relation
\[
\begin{aligned}
\iota\sim\iota'
\quad\Longleftrightarrow\quad
&\text{there exists }h\in O(\Lambda_0)\text{ such that}\\
&h(\iota(L_1))=\iota'(L_1),\qquad
(\iota'^{-1}h\iota)|_{L_2}\in C_{O(L_2)}(f_{L_2}).
\end{aligned}
\]
For fixed $m$ and $d$, let
\[
\calL(L_1,L_2,m,d)=\bigsqcup_{[f_{L_2}]}\calL(L_1,L_2,f_{L_2}),
\]
where the union is over conjugacy-class representatives of order-$m$ isometries with dimension $d$.

The OSCAR routine \texttt{equivariant\_primitive\_extensions} first computes an auxiliary set $\calL_0(S,T,f_T)$ of primitive extensions and retains one fitting ambient isometry for each extension class. This suffices to test existence. In one gluing type, however, the first fitting isometry need not act trivially on $A_{\Lambda_0}$. The program therefore applies the following refinement after the extension has passed the root and symplectic-saturatedness tests.

\begin{lem}\label{lem:stable-fitting-refinement}
Let $(S,T,\iota)$ be an embedded lattice datum in the Laza--Zheng classification with $\rank(S)<20$, and assume that its connected symplectic family has generic index $2$. Write
\[
\iota:S\oplus T\hookrightarrow L\cong\Lambda_0
\]
for the primitive extension, and let $f_L\in O(L)$ be a fitting extension of $f_T$. Let $H_S\subset A_S$ and $H_T\subset A_T$ be the glue subgroups. Then
\[
H_T=A_T,\qquad A_S=H_S\perp R,
\qquad R:=H_S^\perp\cong A_L,
\qquad |R|=3.
\]
If $f_L\notin\widetilde O(L)$, the action of $f_S=f_L|_S$ on $R$ is $-\id_R$. Define
\[
\tau=\id_{H_S}\perp(-\id_R)\in O(q_S).
\]
The isometry $\tau$ lifts to an isometry $u_S\in O(S)$. The isometry $u_S\oplus\id_T$ extends to $u_L\in O(L)$, and
\[
f_L^{\mathrm{st}}=u_Lf_L
\]
lies in $\widetilde O(L)$. Its restriction to $T$ is still $f_T$.

The subgroup
\[
\langle\widetilde O(S),f_L^{\mathrm{st}}|_S\rangle<O(S)
\]
does not depend on the choice of the lift $u_S$.
\end{lem}

\begin{proof}
The generic-index-$2$ criterion in \cite[Proposition-Definitions~4.7, 4.8 and Corollary~4.9]{FWZ26} gives
\[
q_S\cong3^{-1}\oplus(-q_T).
\]
In particular, $|\disc(S)|=3|\disc(T)|$. The discriminant formula for the primitive extension gives $|H_T|=|A_T|$, and hence $H_T=A_T$. The nondegeneracy of the glue form then gives the orthogonal decomposition $A_S=H_S\perp R$, where $|R|=3$. The quotient description of $A_L$ identifies $R$ with $A_L$.

A fitting isometry already has the required action on the glue subgroup. If its action on $A_L$ is nontrivial, it acts as $-\id_R$. Under the decomposition $A_S=H_S\perp R$, the generic non-symplectic involution supplied by the same criterion is induced by an isometry $v_S\in O(S)$. Its action is $-\id_{H_S}\perp\id_R$. Consequently, $u_S=-v_S$ induces $\id_{H_S}\perp(-\id_R)=\tau$.

Since $\tau$ is the identity on $H_S$, the isometry $u_S\oplus\id_T$ preserves the glue and extends to $u_L\in O(L)$. Both $u_L$ and $f_L$ act as $-\id_R$ on the residual summand. Their product therefore acts trivially on $A_L$.

If $u_S'$ is another lift of $\tau$, then $u_S'u_S^{-1}\in\widetilde O(S)$. Therefore the two refined restrictions generate the same subgroup together with $\widetilde O(S)$.
\end{proof}

\begin{algorithm}[htbp]
\caption{Refining a surviving fitting isometry}
\label{alg:stable-fitting-refinement}
\begin{algorithmic}[1]
\Require A fitting extension $(L,S,T,f_L)$ for a Laza--Zheng embedded lattice datum of generic index $2$
\Ensure A stable fitting extension $f_L^{\mathrm{st}}\in\widetilde O(L)$
\If{$f_L\in\widetilde O(L)$}
    \State \Return $f_L$
\EndIf
\State Recover the glue subgroup $H_S\subset A_S$ and compute $R=H_S^\perp$
\State Compute $\tau=\id_{H_S}\perp(-\id_R)\in O(q_S)$
\State Choose $u_S\in O(S)$ inducing $\tau$
\State Extend $u_S\oplus\id_T$ to $u_L\in O(L)$
\State \Return $f_L^{\mathrm{st}}=u_Lf_L$
\end{algorithmic}
\end{algorithm}

\begin{algorithm}[htbp]
\caption{Computing $\calL(S,T,f_T)$}
\label{alg:lattice-enumeration1}
\begin{algorithmic}[1]
\Require Lattices $(S,T)$ and an isometry $f_T\in O(T)$ of order $m$ satisfying the signature condition
\Ensure Representatives of $\calL(S,T,f_T)$, each with a stable fitting isometry
\State $Result\gets\varnothing$
\State Compute $\calL_0(S,T,f_T)$, retaining one fitting isometry $f_{\Lambda_0}$ for each extension class
\ForAll{$(S,T,f_T,\iota,f_{\Lambda_0})\in\calL_0(S,T,f_T)$}
    \If{$f_{\Lambda_0}\notin\widetilde O(\Lambda_0)$ and $|\disc(S)|\neq3|\disc(T)|$}
        \State \textbf{continue}
    \EndIf
    \State $P\gets\ker\Phi_m(f_T)$ and $K\gets\iota(P)^\perp_{\Lambda_0}$
    \If{$K$ contains a short or long root, or $|\widetilde O(K)|\neq|\widetilde O(S)|$}
        \State \textbf{continue}
    \EndIf
    \If{$f_{\Lambda_0}\notin\widetilde O(\Lambda_0)$}
        \State $f_{\Lambda_0}\gets$ Algorithm~\ref{alg:stable-fitting-refinement}$(\Lambda_0,S,T,f_{\Lambda_0})$
    \EndIf
    \State Add $(S,T,f_T,\iota,f_{\Lambda_0})$ to $Result$
\EndFor
\State \Return $Result$
\end{algorithmic}
\end{algorithm}

\begin{lem}\label{lem:algorithm1}
Algorithm~\ref{alg:lattice-enumeration1} computes $\calL(S,T,f_T)$.
\end{lem}

\begin{proof}
The signature condition is imposed on the input. The root test and the equality $|\widetilde O(K)|=|\widetilde O(S)|$ are conditions (3) and (4) of property $\clubsuit$.

For the Laza--Zheng data, a primitive extension has one of the two discriminant types
\[
q_T\cong3^{-1}\oplus(-q_S),\qquad
q_S\cong3^{-1}\oplus(-q_T);
\]
see \cite[Proposition~3.4]{FWZ26}. In the first type, the glue subgroup on the $S$-side is all of $A_S$, while the residual group of order $3$ lies on the $T$-side. The fixed isometry $f_T$ therefore determines the action on $A_{\Lambda_0}$, independently of the chosen fitting extension. Thus an extension class is retained precisely when this action is trivial.

In the second type, the embedded lattice datum has generic index $2$ by \cite[Proposition-Definitions~4.7, 4.8 and Corollary~4.9]{FWZ26}. Lemma~\ref{lem:stable-fitting-refinement} therefore constructs a stable fitting extension whenever the first witness is not stable. The refinement does not change the embedded lattices, $f_T$, $P$, or $K$.

Suppose that two stable fitting extensions represent the same extension class and restrict to the same $f_T$. Their quotient acts trivially on $T$ and on $A_{\Lambda_0}$, so its restriction to $S$ lies in $\widetilde O(S)$. Hence the subgroup generated by $\widetilde O(S)$ and the refined restriction is independent of the fitting witness. The root and group-order tests may therefore be performed before the refinement. It follows that the algorithm retains exactly the extensions satisfying property $\clubsuit$ and attaches a well-defined group to each output.
\end{proof}

\begin{rmk}
The discriminants in the preceding argument are determinants of Gram matrices. The implementation uses their absolute values because the sign displayed by OSCAR depends on the chosen lattice model.
\end{rmk}

\subsection{Identification of Lattice-Theoretic Classes}

For a symplectically saturated output, Proposition~\ref{prop:lattice_family_criteria} recovers $(S,T)$ from the embedded pair $(K,P)$. More precisely, $S$ is the coinvariant lattice of $\widetilde O(K)$ in $K$, and $T=S^\perp_{\Lambda_0}$. Lemma~\ref{lem:group-from-output} recovers the abstract acting group from the induced action on $S$.

Let $\calC(S,T,m,d)$ be the set of lattice-theoretic equivalence classes represented by the resulting tuples $(K,P,f_P,\iota)$.

There is a natural surjection
\[
\calL(S,T,m,d)\longrightarrow\calC(S,T,m,d),
\]
which sends $(S,T,f_T,\iota)$ to the induced $(K,P,f_P,\iota)$. It need not be injective.

\begin{lem}\label{lem:L-to-C}
If $|\calL(S,T,f_T)|\leq1$ for every conjugacy-class representative $f_T$ of order $m$ and dimension $d$, then
\[
\calL(S,T,m,d)\longrightarrow\calC(S,T,m,d)
\]
is bijective.
\end{lem}

\begin{proof}
Suppose two outputs determine the same element of $\calC(S,T,m,d)$. By Proposition~\ref{prop:lattice_family_criteria}(2), an ambient isometry identifies their lattice actions. Since the symplectic group acts trivially on $T$, the restrictions $f_T$ and $f_T'$ are conjugate in $O(T)$.

After replacing both restrictions by the chosen conjugacy-class representative, the two outputs belong to the same set $\calL(S,T,f_T)$. This set has at most one element, so the outputs coincide.
\end{proof}

\begin{lem}\label{lem:auxiliary-KP-test}
Suppose several outputs have isometric lattices $K$ and the same isometry class of $(P,f_P)$. Their lattice-theoretic equivalence classes are in bijection with the root-free primitive extensions of $(K,P,f_P)$ to $\Lambda_0$ that admit a stable fitting isometry and recover the original embedded lattice datum $(S,T,\iota)$.
\end{lem}

\begin{proof}
This is Proposition~\ref{prop:lattice_family_criteria}(2) applied to $(K,P,f_P)$, together with the extension criterion in Lemma~\ref{lem:algorithm1}. In the generic-index-$2$ cases, Algorithm~\ref{alg:stable-fitting-refinement} adjusts the action on the residual discriminant summand. This does not change the embedded lattices or the lattice-theoretic class.
\end{proof}

\subsection{Saturation}

For the computations in this paper, the following consequence of saturation is sufficient. It records the dimension and divisibility restrictions imposed by a proper saturation, together with the inclusion between the acting groups.

\begin{prop}\label{prop:saturation-dimension}
Fix a connected symplectic family. Let $\calF_A$ be a symplectically saturated connected family with acting group $A$ and non-symplectic quotient of order $m$. Put
\[
d=\dim\calF_A.
\]
If $\calF_A$ is not saturated, then its saturation $\calF_{A'}$ satisfies
\[
A\subsetneq A',\qquad m\mid m',\qquad m'>m,
\qquad \dim\calF_{A'}=d,
\]
where $m'=[A':(A')^s]$. Consequently, if the exhaustive output for the same connected symplectic family contains no output of dimension $d$ and order $m'>m$ divisible by $m$, then $\calF_A$ is saturated.
\end{prop}

\begin{proof}
Proposition-Definition~\ref{prop-def:saturation} gives a unique saturation with the same image in $\calM$. Hence the two families have the same dimension, and $A$ is contained in $A'$. Since both families are symplectically saturated inside the same connected symplectic family, their symplectic kernels agree. Their quotients are cyclic and act faithfully on $H^{3,1}$. Thus the quotient of $A$ is a subgroup of the quotient of $A'$. If the saturation is proper, then $m\mid m'$ and $m'>m$.

The candidate list is exhaustive, and Algorithm~\ref{alg:lattice-enumeration2} enumerates all lattice data arising from symplectically saturated actions in the fixed connected symplectic family. Therefore the saturation is represented among the outputs. The last assertion follows by contraposition.
\end{proof}

The converse need not hold: a family with a larger divisible index and the same dimension is not automatically the saturation of a given family. The following restriction argument resolves the only case in which this issue arises.

\begin{lem}\label{lem:equal-dimensional-restriction}
Let $\calF_{A'}$ be a saturated connected family, and let
\[
1\longrightarrow G\longrightarrow A'\longrightarrow\mu_{m'}\longrightarrow1
\]
be its symplectic exact sequence. Let $m$ be a proper divisor of $m'$, and let $A<A'$ be the inverse image of the unique subgroup $\mu_m<\mu_{m'}$. Let $\calF_A$ be the connected family associated with the restricted $A$-action that contains $\calF_{A'}$. Then
\[
\calF_{A'}\subseteq\calF_A,
\qquad
\dim\calF_A\geq\dim\calF_{A'}.
\]

Suppose in addition that the order-$m$ output for the fixed connected symplectic family is unique. If its dimension equals $\dim\calF_{A'}$, then
\[
\calF_A=\calF_{A'},
\]
and $\calF_{A'}$ is the saturation of $\calF_A$.
\end{lem}

\begin{proof}
Every cubic fourfold with an $A'$-action also carries the restricted $A$-action. This gives the inclusion and the dimension inequality.

Let $X$ be a general member of $\calF_{A'}$. Since $\calF_{A'}$ is saturated, $\Aut^s(X)=G$. The local containment result used in the proof of \cite[Proposition~2.5]{yu2020moduli} shows that the automorphism group of a sufficiently general nearby member of $\calF_A$ is conjugate to a subgroup of $\Aut(X)$. Every member of $\calF_A$, however, carries the action of $G$. Hence the generic symplectic automorphism group of $\calF_A$ is $G$. The family $\calF_A$ is therefore symplectically saturated and is represented by an order-$m$ output.

By the observation following the definition of a family, both images are irreducible. They are closed because the forgetful morphisms are finite \cite[Proposition~2.7]{yu2020moduli}. Therefore an inclusion between two such families of the same dimension must be an equality.

If the order-$m$ output is unique and has the same dimension as $\calF_{A'}$, the two images are equal. Since $A<A'$ and $\calF_{A'}$ is saturated, uniqueness of saturation gives the last assertion.
\end{proof}

\subsection{Main Procedure}

Let $m_{\mathrm{generic}}$ denote the generic non-symplectic index of the connected symplectic family. Let $I(S,T)$ be the set of positive integers $m$ such that $m>m_{\mathrm{generic}}$, the generic index divides $m$, and $m$ divides one of the bounds in \cite[Propositions~4.1 and 5.1]{FWZ26}. This finite set contains every possible index larger than the generic index.

\begin{algorithm}[htbp]
\caption{Computing $\bigsqcup_{m\in I(S,T),d}\calL(S,T,m,d)$ and the associated groups}
\label{alg:lattice-enumeration2}
\begin{algorithmic}[1]
\Require Lattices $(S,T)$ and the finite candidate list $I(S,T)$
\Ensure The retained embedded lattice data and the groups $A=\langle\widetilde O(S),f_S\rangle$
\State $Result\gets\varnothing$
\ForAll{$m\in I(S,T)$ in increasing order}
    \If{$\varphi(m)>\rank(T)$}
        \State \textbf{continue}
    \EndIf
    \If{some proper divisor $e\mid m$ belongs to $I(S,T)$ and has empty output}
        \State \textbf{continue}
    \EndIf
    \State Enumerate conjugacy-class representatives $f_T\in O(T)$ of order $m$ satisfying the signature condition
    \ForAll{such $f_T$}
        \State Compute $\calL(S,T,f_T)$ by Algorithm~\ref{alg:lattice-enumeration1}
        \ForAll{$(S,T,f_T,\iota,f_{\Lambda_0})\in\calL(S,T,f_T)$}
            \State $P\gets\ker\Phi_m(f_T)$, $f_P\gets f_T|_P$, and $K\gets\iota(P)^\perp_{\Lambda_0}$
            \State $d\gets\begin{cases}
                \rank(P)/\varphi(m)-2,&m=1,2,\\
                \rank(P)/\varphi(m)-1,&m\geq3
            \end{cases}$
            \State $f_S\gets f_{\Lambda_0}|_S$
            \State Compute $A\gets\langle\widetilde O(S),f_S\rangle<O(S)$
            \State Identify $A$ in GAP and verify $|A|=m|\widetilde O(S)|$
            \State Add $(m,d,|\calL(S,T,f_T)|,A,S,T,f_S,f_T,K,P,f_P,f_{\Lambda_0},\Lambda_0)$ to $Result$
        \EndFor
    \EndFor
\EndFor
\State \Return $Result$
\end{algorithmic}
\end{algorithm}

\begin{prop}\label{prop:main-procedure}
Assume that $(S,T)$ is a Laza--Zheng lattice pair.
\begin{enumerate}
\item Algorithm~\ref{alg:lattice-enumeration2} enumerates every lattice-theoretic datum arising from a symplectically saturated action on a connected component whose non-symplectic index belongs to $I(S,T)$, up to the possible repetitions measured by
\[
\calL(S,T,m,d)\longrightarrow\calC(S,T,m,d).
\]
For each output, it computes the abstract acting group
\[
A=\langle\widetilde O(S),f_S\rangle<O(S).
\]
\item Suppose $|\calL(S,T,f_T)|\leq1$ for every relevant
conjugacy-class representative $f_T$. Then Lemma~\ref{lem:L-to-C}
identifies the outputs with their lattice-theoretic equivalence classes.
\item For a fixed connected symplectic family, an output of order $m$ and dimension $d$ is saturated whenever no output of order $m'>m$ and dimension $d$, with $m\mid m'$, occurs for that family. For a saturated output, the computed group $A$ is the automorphism group of a general member.
\item If the output is empty, then the connected symplectic family has no non-generic index in $I(S,T)$ and no corresponding larger automorphism group.
\end{enumerate}
\end{prop}

\begin{proof}
The cited bounds make $I(S,T)$ exhaustive. The finite-order isometry enumeration is exhaustive up to conjugacy. Lemma~\ref{lem:algorithm1} implements the construction and symplectic-saturatedness criteria, while Lemma~\ref{lem:group-from-output} identifies the group computed on $S$ with the acting group.

The two early exits do not change the output. The inequality $\varphi(m)>\rank(T)$ rules out a primitive $m$-eigenspace in $T$. For the divisor test, suppose that an order-$m$ symplectically saturated family exists and that $e\mid m$. Let $X$ be a general member. The inverse image $A_e$ of the unique subgroup of order $e$ in the non-symplectic quotient defines an order-$e$ family containing $X$.

The local containment result used in the proof of \cite[Proposition~2.5]{yu2020moduli} shows that the automorphism group of a sufficiently general nearby member is conjugate to a subgroup of $\Aut(X)$. This containment respects the character on $H^{3,1}$. Every member still carries the action of the symplectic group, whereas $\Aut^s(X)$ is exactly that group. Hence the restricted family is generically symplectically saturated, and its induced lattice action must occur among the exhaustive order-$e$ outputs. Therefore an empty order-$e$ output implies an empty order-$m$ output whenever $e$ also belongs to the candidate list.

Part (2) is Lemma~\ref{lem:L-to-C}. Part (3) follows from Proposition~\ref{prop:saturation-dimension} and Lemma~\ref{lem:group-from-output}. Part (4) follows from the exhaustiveness of $I(S,T)$.
\end{proof}

\begin{rmk}\label{rmk:computational-range}
The computation starts from the classified pair $(S,T)$. The lattices $K$ and $P$ depend on the chosen isometry $f_T$ and are computed only after an equivariant extension is found. This avoids a separate enumeration of possible pairs $(K,P)$.

When the symplectic group is smaller, $\rank(S)$ decreases and $\rank(T)$ increases. Enumerating the finite-order isometries of $T$ then becomes much more expensive. For this reason, the present calculation is restricted to $\rank(S)\geq15$.
\end{rmk}

The accompanying computational files are listed in \path{anc/README.md}. The OSCAR program and its input are in \path{anc/oscar/oscar_script.jl} and \path{anc/oscar/input.jl}. The ordered search is run by \path{anc/oscar/oscar_run_search.jl}. The complete retained objects are stored in \path{anc/oscar/oscar_script_data.mrdi} and summarized in \path{anc/oscar/oscar_result.md}.

The Gram matrices of $S$ are taken from \cite{hohn2016290}, and the lattices $T$ are reconstructed from their genus data. One can check that the stored representatives have the required signatures and discriminant forms and occur in the required primitive extensions to $\Lambda_0$. For every extension that passes the root and symplectic-saturatedness tests, the program constructs a stable fitting isometry $f_{\Lambda_0}$. It then computes
\[
\langle\widetilde O(S),f_{\Lambda_0}|_S\rangle
\]
and identifies this finite group in GAP \cite{GAP4}. For groups in the SmallGroups library, the saved result records the GAP ID. For larger groups, it records the order and, when available, the value of \texttt{StructureDescription}.

\begin{rmk}\label{rmk:oscar6071}
The computations use OSCAR v1.7.3 with a local correction for the one-unmarked-lattice branch of \texttt{equivariant\_primitive\_extensions}. GitHub issue \#6071 records a false negative in this branch of OSCAR v1.7.3. The correction appears in pull request \#6004, which was merged into the master branch on 15 May 2026. The accompanying program retains the corrected wrapper rather than calling the affected branch directly. See \href{https://github.com/oscar-system/Oscar.jl/issues/6071}{OSCAR issue \#6071}, \href{https://github.com/oscar-system/Oscar.jl/pull/6004}{pull request \#6004}, and \path{anc/oscar/oscar_script.md}.
\end{rmk}

\section{Results and Main Table}\label{sec:results}

\subsection{Computational results}

We ran the algorithm on the $29$ lattice inputs from the Laza--Zheng classification with $15\leq\rank(S)\leq19$ and $I(S,T)\neq\varnothing$. The $S_{3,3}$ input corresponds to two connected symplectic families, so these inputs cover $30$ families. The remaining two families in this rank range, the generic-index-$2$ family for $A_6$ and the family for $T_{48}$, have no non-generic candidate index and require no search.

The two auxiliary equivalence computations used below are reproduced by \path{anc/oscar/oscar_auxiliary_equivalence_check.jl}. Their complete extension objects and comparison data are stored in \path{anc/oscar/oscar_auxiliary_equivalence_data.mrdi}.

For a fixed lattice pair $(S,T)$, Proposition~\ref{prop:ST-families} gives one connected symplectic family, except when $G\cong M_{10}$ or $G\cong S_{3,3}$. We denote a family by $(G,m_{\mathrm{generic}})$ and add a superscript when two connected families have the same lattice pair.

\subsubsection{Rank $19$}

The rank-$19$ cases were classified in \cite{FWZ26}; see also \cite{he2025cubicfourfoldsorder7automorphism}. In rank $19$, the algorithm gives a nonempty output only for $L_2(7)$ and $M_9$.

For $(L_2(7),1)$, there are exactly two outputs. Both have $m=2$, $d=0$, and $|\calL(S,T,f_T)|=1$. They represent distinct lattice-theoretic classes, each with acting group $L_2(7):C_2$. This agrees with the two cases in \cite[Theorem~1.2]{he2025cubicfourfoldsorder7automorphism}.

For $(M_9,1)$, there are two primitive-extension outputs. They satisfy $m=3$, $d=0$, and $|\calL(S,T,f_T)|=2$. Their lattices $K$ are isometric, and their pairs $(P,f_P)$ are also isometric. Repeating the primitive-extension computation for this fixed pair gives exactly one root-free class admitting a stable fitting isometry. Thus the two outputs represent a single lattice-theoretic class. Its computed automorphism group is $M_9:C_3$, in agreement with \cite{FWZ26}. The files cited above record the auxiliary computation.

\subsubsection{The two $S_{3,3}$ families}

The same lattice pair $(S,T)$ gives two connected symplectic families of generic index $2$. Their generic automorphism groups are $S_{3,3}\times C_2$ and $N_{72}$; see \cite{KOIKE202512,KOIKE2026}.

The algorithm gives two outputs with $m=6$ and $d=1$. Their GAP IDs are $(216,170)$ and $(216,157)$, respectively. The corresponding groups are
\[
S_{3,3}\times C_6
\qquad\text{and}\qquad
N_{72}\times C_3.
\]
The group $S_{3,3}\times C_2$ is not isomorphic to a subgroup of $N_{72}\times C_3$, and $N_{72}$ is not isomorphic to a subgroup of $S_{3,3}\times C_6$. Hence the output with GAP ID $(216,157)$ lies over the $N_{72}$ family, while the output with GAP ID $(216,170)$ lies over the other family. The two outputs are therefore distinguished by their automorphism groups. Both are saturated by Proposition~\ref{prop:saturation-dimension}.

\subsubsection{The family $(A_{3,3},2)$}

There are two outputs with $m=6$, $d=2$, and $|\calL(S,T,6,2)|=2$. They have isometric lattices $K$ and the same isometry class of $(P,f_P)$. The auxiliary extension computation gives two classes. The traces of the two induced actions on $S$ are $-4$ and $-1$, so the outputs cannot be conjugate. Lemma~\ref{lem:auxiliary-KP-test} therefore shows that they represent two distinct lattice-theoretic classes. The files cited above record this computation.

Lemma~\ref{lem:group-from-output}, together with the order check in Remark~\ref{rmk:group-order-on-S}, identifies their abstract groups from the induced actions on $S$. Proposition~\ref{prop:saturation-dimension} shows that both outputs are saturated, so these are the generic automorphism groups on every corresponding component.

\subsubsection{The remaining cases with
$15\leq\rank(S)\leq18$}

For each fixed connected symplectic family, we consider only the pairs $(m,d)$ for which the output is nonempty.

In every remaining case with $15\leq\rank(S)\leq18$, each relevant conjugacy-class representative satisfies $|\calL(S,T,f_T)|=1$. By Lemma~\ref{lem:L-to-C}, the outputs therefore represent distinct lattice-theoretic equivalence classes. The candidate list is exhaustive by \cite[Propositions~4.1 and 5.1]{FWZ26}. Proposition~\ref{prop:main-procedure} then determines the abstract groups and their indices.

Except for the family with $G\cong F_{21}$, no two output pairs $(m,d)$ and $(m',d)$ satisfy $m\mid m'$ with $m<m'$.

By Proposition~\ref{prop:saturation-dimension}, an output is saturated unless a larger divisible order occurs in the same dimension. Within a fixed connected symplectic family, this happens only for the zero-dimensional outputs of orders $3$ and $6$ for $F_{21}$. Lemma~\ref{lem:equal-dimensional-restriction} resolves this case.

For $G\cong F_{21}$, the three outputs have $(m,d)=(2,1),(3,0),(6,0)$. The order-$2$ and order-$6$ outputs are saturated by Proposition~\ref{prop:saturation-dimension}, so their computed groups are the generic automorphism groups on every corresponding component.

It remains to consider the order-$3$ output. Restrict the saturated order-$6$ action to the inverse image of the unique subgroup $\mu_3\subset\mu_6$. This gives an order-$3$ family containing the order-$6$ family. By Lemma~\ref{lem:L-to-C}, the order-$3$ output is unique, and both families have dimension $0$. Lemma~\ref{lem:equal-dimensional-restriction} therefore shows that they have the same image in $\calM$. Thus the order-$3$ family is not saturated; its saturation is the order-$6$ family.

Thus the order-$3$ action does not give a new automorphism group. The automorphism groups for the $F_{21}$ symplectic family are the generic group and the two saturated groups of indices $2$ and $6$, as recorded in Table~\ref{table: main}.

The same procedure identifies the group for every remaining saturated output. The resulting groups and indices are recorded in Table~\ref{table: main}.

\subsection{Proof of the main theorem}

\begin{proof}[Proof of Theorem~\ref{thm:main}]
For $\rank(S)=20$, the connected families are points. Their automorphism groups are given by \cite[Theorem~1.8]{laza2022automorphisms}; see also \cite{KOIKE202512,KOIKE2026}.

For $\rank(S)=19$, the classification follows from \cite{FWZ26} and is reproduced by the computations above. Now assume $15\leq\rank(S)\leq18$. The candidate indices are exhaustive by \cite[Propositions~4.1 and 5.1]{FWZ26}. Proposition~\ref{prop:main-procedure} and Lemmas~\ref{lem:L-to-C} and \ref{lem:auxiliary-KP-test} enumerate and identify all lattice-theoretic data arising from symplectically saturated actions. Lemma~\ref{lem:group-from-output} determines the abstract group attached to each output.

Proposition~\ref{prop:saturation-dimension} and Lemma~\ref{lem:equal-dimensional-restriction} determine which outputs are saturated. For a saturated output, the group attached to it is the automorphism group of a general member. Conversely, Proposition~\ref{prop:construct-family} realizes each retained root-free lattice action on smooth cubic fourfolds. Hence the preceding subsections give exactly the possible groups and indices recorded in Table~\ref{table: main}. For a real character, the two type IV components have the same group, index, dimension, and saturation behavior by Proposition~\ref{prop:lattice-equivalence}.
\end{proof}

There are two further lower-rank families for which the full
automorphism groups are known. These cases are not part of
Theorem~\ref{thm:main}. For the second symplectic family with
$\Aut^s(X)\cong S_3$, whose generic non-symplectic index is $2$, the
possible indices are $m\in\{2,4,6,8,12,24\}$ by
\cite[Remark~4.10 and Theorem~1.1]{FWZ26}. To determine the full group,
put $H=\Aut(X)$ and $N=\Aut^s(X)\cong S_3$. Since $Z(N)=1$ and
$\operatorname{Out}(N)=1$, conjugation gives
\[
H=N C_H(N),\qquad N\cap C_H(N)=1.
\]
The centralizer $C_H(N)$ maps isomorphically to
$H/N\cong C_m$. Hence $\Aut(X)\cong S_3\times C_m$.

If $\Aut^s(X)=1$, the character on $H^{3,1}(X)$ is faithful, so
$\Aut(X)\cong C_m$. By \cite[Proposition~7.2]{FWZ26}, the possible
values of $m$ are precisely the positive integers dividing $32$ or
$48$, and every such value occurs. Together with
Theorem~\ref{thm:main}, these two cases give the complete list of
abstract automorphism groups and non-symplectic indices for $42$ of
the $48$ connected symplectic families.

\subsection{Main table}

Table~\ref{table: main} combines the classification in
Theorem~\ref{thm:main} with the lower-rank data from \cite{FWZ26}. The
part with $\rank(S)\geq15$ gives the classification proved in this
paper. The final eight rows correspond to the connected symplectic
families with $\rank(S)<15$. The second $S_3$ row and the row with
trivial symplectic automorphism group give the complete group-and-index
lists discussed above. The other six rows record only the generic
indices and the YYZ bounds. None of the eight lower-rank rows is part
of Theorem~\ref{thm:main}.

There are $40$ connected symplectic families with $\rank(S)\geq15$. Their generic full groups give $40$ rows in the table. The search produces $38$ further raw records. The two records for $M_9$ represent one lattice-theoretic class, and the order-$3$ record for $F_{21}$ has the order-$6$ family as its saturation. Thus there are $36$ further saturated lattice-theoretic records, and the part with $\rank(S)\geq15$ contains $76$ rows. We do not interpret this number as a count of equivalence classes of connected families with group actions. The eight lower-rank rows are not included in this count.

The column ``YYZ bounds'' gives the Yang--Yu--Zhu divisibility bounds recorded in \cite{FWZ26}. The column ``generic index'' gives the generic non-symplectic index in the connected symplectic family.

Each row in the part with $\rank(S)\geq15$ records the group-and-index data associated with a saturated lattice-theoretic class. Its generic symplectic automorphism group is isomorphic to the group in the column ``$\Aut^s(X)$'', and its generic automorphism group is isomorphic to the group in the column ``$\Aut(X)$''. In the real-character cases, one row may correspond to two complex-conjugate connected components, which are not distinguished here. Repeated rows with the same displayed numerical and group-theoretic data represent distinct lattice-theoretic classes, as distinguished by their embedded lattice data or by the preceding discussion.

In the part with $\rank(S)\geq15$, the column ``index'' gives the non-symplectic index recorded by the lattice-theoretic class, and the column ``$\dim\calF$'' gives the common dimension of its associated component or components. The column ``GAP ID of $\Aut(X)$'' gives the GAP ID of $\Aut(X)$ when its order is at most $2000$.

In the two lower-rank rows discussed above, the merged entry spans the
columns ``index'', ``$\Aut(X)$'', and ``GAP ID of $\Aut(X)$'' and gives
the complete group-and-index list. In the other six lower-rank rows, a
blank entry means that the corresponding datum is not classified or
recorded here; it does not mean that no such case exists.

{\footnotesize
\setlength{\tabcolsep}{2pt}
\begin{longtable}{|c|c|c|c|c|ccc|c|}
\hline
$\rank(S)$    & $G=\Aut^s$  & $\ord(G)$   & $ \begin{matrix}\mathrm{YYZ}\\ \mathrm{bounds} \end{matrix}$ & $ \begin{matrix}\mathrm{generic}\\ \mathrm{index} \end{matrix}$ & \multicolumn{1}{c|}{index} & \multicolumn{1}{c|}{$\Aut(X)$} & GAP ID of $\Aut(X)$   & $\dim\calF$ \\ \hline
$20$& $C_3^4:A_6$ & $29160$     & $6$    & $6$ & \multicolumn{1}{c|}{$6$}   & \multicolumn{1}{c|}{$C_3^5:S_6$}     &     & $0$   \\ \hline
\multirow[t]{2}{*}{$20$} & \multirow[t]{2}{*}{$A_7$}  & \multirow[t]{2}{*}{$2520$} & \multirow[t]{2}{*}{$2$}     & $1$ & \multicolumn{1}{c|}{$1$}   & \multicolumn{1}{c|}{$A_7$}     &     & $0$   \\ \cline{5-9} 
    & & &  & $2$ & \multicolumn{1}{c|}{$2$}   & \multicolumn{1}{c|}{$S_7$}     &     & $0$   \\ \hline
$20$& $3^{1+4}:C_2.C_2^2$     & $1944$& $4$    & $4$ & \multicolumn{1}{c|}{$4$}   & \multicolumn{1}{c|}{$3^{1+4}:(C_4^2:C_2)$} &     & $0$   \\ \hline
\multirow[t]{2}{*}{$20$} & \multirow[t]{2}{*}{$M_{10}$}     & \multirow[t]{2}{*}{$720$}  & \multirow[t]{2}{*}{$1$}     & $1$ & \multicolumn{1}{c|}{$1$}   & \multicolumn{1}{c|}{$M_{10}$}  & (720,765) & $0$   \\ \cline{5-9} 
    & & &  & $1$ & \multicolumn{1}{c|}{$1$}   & \multicolumn{1}{c|}{$M_{10}$}  & (720,765) & $0$   \\ \hline
$20$& $L_{2}(11)$ & $660$ & $3$    & $3$ & \multicolumn{1}{c|}{$3$}   & \multicolumn{1}{c|}{$L_2(11)\times C_3$}   & (1980,57) & $0$   \\ \hline
$20$& $A_{3,5}$   & $360$ & $6$    & $6$ & \multicolumn{1}{c|}{$6$}   & \multicolumn{1}{c|}{$S_{3,5}\times C_3$}   &     & $0$   \\ \hline
$19$& $3^{1+4}:C_2.C_2$ & $972$ & $6$    & $2$ & \multicolumn{1}{c|}{$2$}   & \multicolumn{1}{c|}{$3^{1+4}:D_8$}   & (1944,3536) & $1$   \\ \hline
\multirow[t]{2}{*}{$19$} & \multirow[t]{2}{*}{$A_{6}$}& \multirow[t]{2}{*}{$360$}  & \multirow[t]{2}{*}{$2$}     & $1$ & \multicolumn{1}{c|}{$1$}   & \multicolumn{1}{c|}{$A_6$}     & (360,118) & $1$   \\ \cline{5-9} 
    & & &  & $2$ & \multicolumn{1}{c|}{$2$}   & \multicolumn{1}{c|}{$S_6$}     & (720,763) & $1$   \\ \hline
\multirow[t]{3}{*}{$19$} & \multirow[t]{3}{*}{$L_{2}(7)$}   & \multirow[t]{3}{*}{$168$}  & \multirow[t]{3}{*}{$2$}     & \multirow[t]{3}{*}{$1$}  & \multicolumn{1}{c|}{$1$}   & \multicolumn{1}{c|}{$L_2(7)$}  & (168,42)  & $1$   \\ \cline{6-9} 
    & & &  &     & \multicolumn{1}{c|}{$2$}   & \multicolumn{1}{c|}{$L_2(7):C_2$}    & (336,208) & $0$   \\ \cline{6-9} 
    & & &  &     & \multicolumn{1}{c|}{$2$}   & \multicolumn{1}{c|}{$L_2(7):C_2$}    & (336,208) & $0$   \\ \hline
\multirow[t]{2}{*}{$19$} & \multirow[t]{2}{*}{$S_{5}$}& \multirow[t]{2}{*}{$120$}  & \multirow[t]{2}{*}{$6$}     & $1$ & \multicolumn{1}{c|}{$1$}   & \multicolumn{1}{c|}{$S_5$}     & (120,34)  & $1$   \\ \cline{5-9} 
    & & &  & $2$ & \multicolumn{1}{c|}{$2$}   & \multicolumn{1}{c|}{$S_5\times C_2$} & (240,189) & $1$   \\ \hline
\multirow[t]{2}{*}{$19$} & \multirow[t]{2}{*}{$M_{9}$}& \multirow[t]{2}{*}{$72$}   & \multirow[t]{2}{*}{$3$}     & \multirow[t]{2}{*}{$1$}  & \multicolumn{1}{c|}{$1$}   & \multicolumn{1}{c|}{$M_9$}     & (72,41)   & $1$   \\ \cline{6-9} 
    & & &  &     & \multicolumn{1}{c|}{$3$}   & \multicolumn{1}{c|}{$M_9:C_3$} & (216,153) & $0$   \\ \hline
$19$& $N_{72}$    & $72$  & $6$    & $2$ & \multicolumn{1}{c|}{$2$}   & \multicolumn{1}{c|}{$N_{72}\times C_2$}    & (144,186) & $1$   \\ \hline
$19$& $T_{48}$    & $48$  & $1$    & $1$ & \multicolumn{1}{c|}{$1$}   & \multicolumn{1}{c|}{$T_{48}$}  &   (48,29)   & $1$   \\ \hline
\multirow[t]{4}{*}{$18$} & \multirow[t]{4}{*}{$3^{1+4}:C_2$}& \multirow[t]{4}{*}{$486$}  & \multirow[t]{4}{*}{$12$}    & \multirow[t]{4}{*}{$2$}  & \multicolumn{1}{c|}{$2$}   & \multicolumn{1}{c|}{$3^{1+4}:C_2^2$} & (972,812) & $2$   \\ \cline{6-9} 
    & & &  &     & \multicolumn{1}{c|}{$4$}   & \multicolumn{1}{c|}{$3^{1+4}:(C_4\times C_2)$}   & (1944,3493)     & $1$   \\ \cline{6-9} 
    & & &  &     & \multicolumn{1}{c|}{$6$}   & \multicolumn{2}{c|}{$((C_3\times(C_3^3:C_3)):C_3):(C_2\times C_2)$}& $1$   \\ \cline{6-9} 
    & & &  &     & \multicolumn{1}{c|}{$12$}  & \multicolumn{2}{c|}{$((C_3 \times ((C_3^3) : C_3)) : C_3) : (C_4 \times C_2)$} & $0$   \\ \hline
\multirow[t]{4}{*}{$18$} & \multirow[t]{4}{*}{$A_{4,3}$}    & \multirow[t]{4}{*}{$72$}   & \multirow[t]{4}{*}{$6$}     & \multirow[t]{2}{*}{$1$}  & \multicolumn{1}{c|}{$1$}   & \multicolumn{1}{c|}{$A_{4,3}$} & (72,43)   & $2$   \\ \cline{6-9} 
    & & &  &     & \multicolumn{1}{c|}{$2$}   & \multicolumn{1}{c|}{$A_{4,3}\times C_2$}   & (144,189) & $1$   \\ \cline{5-9} 
    & & &  & \multirow[t]{2}{*}{$2$}  & \multicolumn{1}{c|}{$2$}   & \multicolumn{1}{c|}{$S_{4,3}$} & (144,183) & $2$   \\ \cline{6-9} 
    & & &  &     & \multicolumn{1}{c|}{$6$}   & \multicolumn{1}{c|}{$S_{4,3}\times C_3$}   & (432,745) & $1$   \\ \hline
\multirow[t]{4}{*}{$18$} & \multirow[t]{4}{*}{$A_{5}$}& \multirow[t]{4}{*}{$60$}   & \multirow[t]{4}{*}{$6$}     & \multirow[t]{2}{*}{$1$}  & \multicolumn{1}{c|}{$1$}   & \multicolumn{1}{c|}{$A_5$}     & (60,5)    & $2$   \\ \cline{6-9} 
    & & &  &     & \multicolumn{1}{c|}{$3$}   & \multicolumn{1}{c|}{$A_5\times C_3\cong \GL(2,4)$}     & (180,19)  & $1$   \\ \cline{5-9} 
    & & &  & \multirow[t]{2}{*}{$2$}  & \multicolumn{1}{c|}{$2$}   & \multicolumn{1}{c|}{$S_5$}     & (120,34)  & $2$   \\ \cline{6-9} 
    & & &  &     & \multicolumn{1}{c|}{$6$}   & \multicolumn{1}{c|}{$S_5\times C_3$} & (360,119) & $1$   \\ \hline
\multirow[t]{2}{*}{$18$} & \multirow[t]{2}{*}{$C_3^{2}.C_4$}& \multirow[t]{2}{*}{$36$}   & \multirow[t]{2}{*}{$6$}     & $1$ & \multicolumn{1}{c|}{$1$}   & \multicolumn{1}{c|}{$C_3^{2}.C_4$}   & (36,9)    & $2$   \\ \cline{5-9} 
    & & &  & $2$ & \multicolumn{1}{c|}{$2$}   & \multicolumn{1}{c|}{$N_{72}$}  & (72,40)   & $2$   \\ \hline
\multirow[t]{4}{*}{$18$} & \multirow[t]{4}{*}{$S_{3,3}$}    & \multirow[t]{4}{*}{$36$}   & \multirow[t]{4}{*}{$6$}     & \multirow[t]{2}{*}{$2$}  & \multicolumn{1}{c|}{$2$}   & \multicolumn{1}{c|}{$S_{3,3}\times C_2$}   & (72,46)   & $2$   \\ \cline{6-9} 
    & & &  &     & \multicolumn{1}{c|}{$6$}   & \multicolumn{1}{c|}{$S_{3,3}\times C_6$}   & (216,170) & $1$   \\ \cline{5-9} 
    & & &  & \multirow[t]{2}{*}{$2$}  & \multicolumn{1}{c|}{$2$}   & \multicolumn{1}{c|}{$N_{72}$}  & (72,40)   & $2$   \\ \cline{6-9} 
    & & &  &     & \multicolumn{1}{c|}{$6$}   & \multicolumn{1}{c|}{$N_{72}\times C_3$}    & (216,157) & $1$   \\ \hline
\multirow[t]{3}{*}{$18$} & \multirow[t]{3}{*}{$F_{21}$}     & \multirow[t]{3}{*}{$21$}   & \multirow[t]{3}{*}{$6$}     & \multirow[t]{3}{*}{$1$}  & \multicolumn{1}{c|}{$1$}   & \multicolumn{1}{c|}{$F_{21}$}  & (21,1)    & $2$   \\ \cline{6-9} 
    & & &  &     & \multicolumn{1}{c|}{$2$}   & \multicolumn{1}{c|}{$C_7:C_6$} & (42,1)    & $1$   \\ \cline{6-9} 
    & & &  &     & \multicolumn{1}{c|}{$6$}   & \multicolumn{1}{c|}{$(C_7:C_6)\times C_3$} & (126,7)   & $0$   \\ \hline
$18$& $\mathrm{Hol}(5)$ & $20$  & $6$    & $1$ & \multicolumn{1}{c|}{$1$}   & \multicolumn{1}{c|}{$\mathrm{Hol}(5)$}     & (20,3)    & $2$   \\ \hline
\multirow[t]{2}{*}{$18$} & \multirow[t]{2}{*}{$\mathrm{QD}_{16}$} & \multirow[t]{2}{*}{$16$}   & \multirow[t]{2}{*}{$2$}     & \multirow[t]{2}{*}{$1$}  & \multicolumn{1}{c|}{$1$}   & \multicolumn{1}{c|}{$\mathrm{QD}_{16}$}    & (16,8)    & $2$   \\ \cline{6-9} 
    & & &  &     & \multicolumn{1}{c|}{$2$}   & \multicolumn{1}{c|}{$(C_8\times C_2): C_2$}& (32,42)   & $1$   \\ \hline
\multirow[t]{3}{*}{$17$} & \multirow[t]{3}{*}{$S_{4}$}& \multirow[t]{3}{*}{$24$}   & \multirow[t]{3}{*}{$6$}     & \multirow[t]{2}{*}{$1$}  & \multicolumn{1}{c|}{$1$}   & \multicolumn{1}{c|}{$S_4$}     & (24,12)   & $3$   \\ \cline{6-9} 
    & & &  &     & \multicolumn{1}{c|}{$2$}   & \multicolumn{1}{c|}{$S_4\times C_2$} & (48,48)   & $2$   \\ \cline{5-9} 
    & & &  & $2$ & \multicolumn{1}{c|}{$2$}   & \multicolumn{1}{c|}{$S_4\times C_2$} & (48,48)   & $3$   \\ \hline
\multirow[t]{4}{*}{$17$} & \multirow[t]{4}{*}{$Q_{8}$}& \multirow[t]{4}{*}{$8$}    & \multirow[t]{4}{*}{$3$,$4$} & \multirow[t]{4}{*}{$1$}  & \multicolumn{1}{c|}{$1$}   & \multicolumn{1}{c|}{$Q_8$}     & (8,4)     & $3$   \\ \cline{6-9} 
    & & &  &     & \multicolumn{1}{c|}{$2$}   & \multicolumn{1}{c|}{$(C_4\times C_2):C_2$} & (16,13)   & $2$   \\ \cline{6-9} 
    & & &  &     & \multicolumn{1}{c|}{$3$}   & \multicolumn{1}{c|}{$\SL(2,3)$}& (24,3)    & $1$   \\ \cline{6-9} 
    & & &  &     & \multicolumn{1}{c|}{$4$}   & \multicolumn{1}{c|}{$C_4^2: C_2$}    & (32,11)   & $1$   \\ \hline
\multirow[t]{6}{*}{$16$} & \multirow[t]{6}{*}{$A_{3,3}$}    & \multirow[t]{6}{*}{$18$}   & \multirow[t]{6}{*}{$12$}    & \multirow[t]{3}{*}{$1$}  & \multicolumn{1}{c|}{$1$}   & \multicolumn{1}{c|}{$A_{3,3}$} & (18,4)    & $4$   \\ \cline{6-9} 
    & & &  &     & \multicolumn{1}{c|}{$2$}   & \multicolumn{1}{c|}{$A_{3,3}\times C_2$}   & (36,13)   & $2$   \\ \cline{6-9} 
    & & &  &     & \multicolumn{1}{c|}{$3$}   & \multicolumn{1}{c|}{$C_3^2:C_6$}     & (54,5)    & $1$   \\ \cline{5-9} 
    & & &  & \multirow[t]{3}{*}{$2$}  & \multicolumn{1}{c|}{$2$}   & \multicolumn{1}{c|}{$S_{3,3}$}   & (36,10)   & $4$   \\ \cline{6-9} 
    & & &  &     & \multicolumn{1}{c|}{$6$}   & \multicolumn{1}{c|}{$S_{3,3}\times C_3$}   & (108,38)  & $2$   \\ \cline{6-9} 
    & & &  &     & \multicolumn{1}{c|}{$6$}   & \multicolumn{1}{c|}{$S_{3,3}\times C_3$}   & (108,38)  & $2$   \\ \hline
\multirow[t]{7}{*}{$16$} & \multirow[t]{7}{*}{$D_{12}$}     & \multirow[t]{7}{*}{$12$}   & \multirow[t]{7}{*}{$12$}    & \multirow[t]{5}{*}{$1$}  & \multicolumn{1}{c|}{$1$}   & \multicolumn{1}{c|}{$D_{12}$}  & (12,4)    & $4$   \\ \cline{6-9} 
    & & &  &     & \multicolumn{1}{c|}{$2$}   & \multicolumn{1}{c|}{$D_{12}\times C_2$}    & (24,14)   & $2$   \\ \cline{6-9} 
    & & &  &     & \multicolumn{1}{c|}{$2$}   & \multicolumn{1}{c|}{$(C_6\times C_2):C_2$} & (24,8)    & $2$   \\ \cline{6-9} 
    & & &  &     & \multicolumn{1}{c|}{$3$}   & \multicolumn{1}{c|}{$D_{12}\times C_3$}    & (36,12)   & $2$   \\ \cline{6-9} 
    & & &  &     & \multicolumn{1}{c|}{$6$}   & \multicolumn{1}{c|}{$((C_6\times C_2):C_2)\times C_3$}   & (72,30)   & $1$   \\ \cline{5-9} 
    & & &  & \multirow[t]{2}{*}{$2$}  & \multicolumn{1}{c|}{$2$}   & \multicolumn{1}{c|}{$D_{12}\times C_2$}    & (24,14)   & $4$   \\ \cline{6-9} 
    & & &  &     & \multicolumn{1}{c|}{$6$}   & \multicolumn{1}{c|}{$S_3\times C_6\times C_2$}   & (72,48)   & $2$   \\ \hline
\multirow[t]{5}{*}{$16$} & \multirow[t]{5}{*}{$A_{4}$}& \multirow[t]{5}{*}{$12$}   & \multirow[t]{5}{*}{$6$}     & \multirow[t]{3}{*}{$1$}  & \multicolumn{1}{c|}{$1$}   & \multicolumn{1}{c|}{$A_4$}     & (12,3)    & $4$   \\ \cline{6-9} 
    & & &  &     & \multicolumn{1}{c|}{$2$}   & \multicolumn{1}{c|}{$A_4\times C_2$} & (24,13)   & $3$   \\ \cline{6-9} 
    & & &  &     & \multicolumn{1}{c|}{$3$}   & \multicolumn{1}{c|}{$A_4\times C_3$} & (36,11)   & $2$   \\ \cline{5-9} 
    & & &  & \multirow[t]{2}{*}{$2$}  & \multicolumn{1}{c|}{$2$}   & \multicolumn{1}{c|}{$S_4$}     & (24,12)   & $4$   \\ \cline{6-9} 
    & & &  &     & \multicolumn{1}{c|}{$6$}   & \multicolumn{1}{c|}{$S_4\times C_3$} & (72,42)   & $2$   \\ \hline
\multirow[t]{3}{*}{$16$} & \multirow[t]{3}{*}{$D_{10}$}     & \multirow[t]{3}{*}{$10$}   & \multirow[t]{3}{*}{$12$}    & \multirow[t]{3}{*}{$1$}  & \multicolumn{1}{c|}{$1$}   & \multicolumn{1}{c|}{$D_{10}$}  & (10,1)    & $4$   \\ \cline{6-9} 
    & & &  &     & \multicolumn{1}{c|}{$2$}   & \multicolumn{1}{c|}{$D_{20}$}  & (20,4)    & $2$   \\ \cline{6-9} 
    & & &  &     & \multicolumn{1}{c|}{$3$}   & \multicolumn{1}{c|}{$D_{10}\times C_3$}    & (30,2)    & $2$   \\ \hline
\multirow[t]{4}{*}{$15$} & \multirow[t]{4}{*}{$D_{8}$}& \multirow[t]{4}{*}{$8$}    & \multirow[t]{4}{*}{$4$,$6$} & \multirow[t]{4}{*}{$1$}  & \multicolumn{1}{c|}{$1$}   & \multicolumn{1}{c|}{$D_8$}     & (8,3)     & $5$   \\ \cline{6-9} 
    & & &  &     & \multicolumn{1}{c|}{$2$}   & \multicolumn{1}{c|}{$D_8\times C_2$} & (16,11)   & $4$   \\ \cline{6-9} 
    & & &  &     & \multicolumn{1}{c|}{$2$}   & \multicolumn{1}{c|}{$(C_4\times C_2):C_2$} & (16,13)   & $3$   \\ \cline{6-9} 
    & & &  &     & \multicolumn{1}{c|}{$2$}   & \multicolumn{1}{c|}{$D_{16}$}  & (16,7)    & $2$   \\ \hline
$14$& $C_4$ & $4$   & $8$,$12$     & $1$ & \multicolumn{1}{c|}{}& \multicolumn{1}{c|}{}    &     & \\ \hline
\multirow[t]{2}{*}{$14$} & \multirow[t]{2}{*}{$S_3$}  & \multirow[t]{2}{*}{$6$}    & \multirow[t]{2}{*}{$24$}    & $1$ & \multicolumn{1}{c|}{}& \multicolumn{1}{c|}{}    &     & \\ \cline{5-9}
    & & &  & $2$ & \multicolumn{3}{c|}{$\Aut(X)\cong S_3\times C_m,\quad 2\mid m\mid24$}   & \\ \hline
$12$& $C_2^{2}$   & $4$   & $12$   & $1$ & \multicolumn{1}{c|}{}& \multicolumn{1}{c|}{}    &     & \\ \hline
\multirow[t]{2}{*}{$12$} & \multirow[t]{2}{*}{$C_3$}  & \multirow[t]{2}{*}{$3$}    & \multirow[t]{2}{*}{$16$,$24$}     & $1$ & \multicolumn{1}{c|}{}& \multicolumn{1}{c|}{}    &     & \\ \cline{5-9}
    & & &  & $2$ & \multicolumn{1}{c|}{}& \multicolumn{1}{c|}{}    &     & \\ \hline
$8$ & $C_2$ & $2$   & $16$,$24$    & $1$ & \multicolumn{1}{c|}{}& \multicolumn{1}{c|}{}    &     & \\ \hline
$0$ & $C_1$ & $1$   & $32$,$48$    & $1$ & \multicolumn{3}{c|}{$\Aut(X)\cong C_m,\quad m\mid32\text{ or }m\mid48$}& \\ \hline
\caption{Group-and-index classification for $42$ connected symplectic families, together with bounds for the remaining six}
\label{table: main}
\end{longtable}
}

\bibliography{reference}
\end{document}